\documentclass{amsart}
\usepackage{amsmath,amssymb}
  \usepackage{paralist}
 \usepackage[colorlinks=true]{hyperref}
\hypersetup{urlcolor=blue, citecolor=red}

\newtheorem{theorem}{Theorem}[section]
\newtheorem{corollary}{Corollary}

\newtheorem{lemma}[theorem]{Lemma}
\newtheorem{proposition}{Proposition}

\theoremstyle{definition}

\newtheorem{remark}{Remark}

\newcommand{\eps}[1]{{#1}_{\varepsilon}}

\makeatletter
\newcommand{\en}     {{\varepsilon_n}}
\renewcommand{\epsilon}  {\varepsilon}
\newcommand{\orbit}{\gamma}
\newcommand{\cc}[1]{}
\newcommand{\one}{{\bf1}}
\newcommand{\sym} {\mathfrak{s}}
\newcommand{\F} {{ \mathcal F }}

\newcommand{\R} {\mathbf{R}}

\newcommand{\Z} {\mathbf{Z}}
\renewcommand{\P} {\mathbf{P}}
\newcommand{\E} {\mathbf{E}}

\newcommand{\avg}[1]{\mathcal{A}(#1)}

\newcommand{\hf}    {{\textstyle \frac12}}

\newcommand{\e}     {\varepsilon}

\newcommand{\sign}   {\mathrm{sign}}
\newcommand{\sn}   {\mathfrak{sn}}
\newcommand{\s}   {\mathfrak{s}}

\renewcommand{\stop} {\kappa^\epsilon_M}

\newcommand{\stopL} {\kappa_M}
\newcommand{\bpf}[1][Proof]{{\noindent {\sc #1: }}}
\newcommand{\epf}   {{\hfill $\square$}}

\newcommand{\eqdef}{\!\overset{\text{\tiny def}}{=}\!}
\newcommand{\BB}    {\mathcal{B}}

\usepackage{tikz}

\title[Invariant measure selection by
noise]{Invariant measure selection by
noise\\ An Example}

\author[Jonathan Mattingly and Etienne Pardoux]{}



 \email{jonm@math.duke.edu}
 \email{etienne.pardoux@univ-amu.fr}

\begin{document}

\begin{abstract}
  We consider a deterministic system with two conserved quantities and infinity many 
  invariant measures. However the systems possess a unique invariant 
  measure when enough stochastic forcing and balancing dissipation are 
  added. We then show that as the forcing and dissipation are removed 
  a unique limit of the deterministic system is selected. The exact 
  structure of the limiting measure depends on the specifics of the stochastic forcing. 
\end{abstract}
\maketitle

\centerline{\scshape Jonathan C. Mattingly}
\medskip
{\footnotesize

 \centerline{Mathematics Department and Department of Statisical Science}
   \centerline{Duke University, Box 90320 }
   \centerline{Durham, NC 27708-0320, USA}
} 

\medskip

\centerline{\scshape Etienne Pardoux}
\medskip
{\footnotesize
 
 \centerline{ Laboratoire d'Analyse, Topologie, Probabilit\'es}
   \centerline{Universit\'e de Provence
39, rue F. Joliot-Curie}
   \centerline{F-13453 Marseille cedex 13, France}
}

\bigskip
\begin{center}
  \textit{Dedicated to the memory of Jos\'e Real}
\end{center}

\section{Introduction}
\label{sec:introduction}

There is much interest in the regularizing effects of noise on the
longtime dynamics. One often speaks informally of adding a balancing
noise and dissipation to a dynamical system with many invariant
measures and then studying the zero noise/dissipation limit as a way
of selecting the ``physically relevant'' invariant measure.

There are a number of settings where such a procedure is fairly well
understood. In the case of a Hamiltonian or gradient system with
sufficiently non-degenerate noise, Wentzell-Freidlin theory gives a
rather complete description of the effective limiting
dynamics \cite{FreidlinWentzell} in terms of a limiting ``slow'' system
derived through a quasi-potential and deterministic averaging. In the
gradient case the stochastic invariant measures concentrate on the
attracting structures of the dynamics. In the Hamiltonian setting,
Wentzell-Freidlin theory considers the slow dynamics of the
conserved quantity (the Hamiltonian)
when the system is subject to noise. It is the zero noise limit of
these dynamics which decides which mixture of the Hamiltonian
invariant measures is selected in the zero noise limit.

In the case of system with an underlying hyperbolic structure, such as
Axiom A,  it is known that the zero noise limit of random perturbations
selects a canonical SRB/''physical
measure''\cite{sinai1,sinai2,Ruelle,Kifer}. This relies fundamentally
on the expansion/contraction properties of the underlying
deterministic dynamical system. See \cite{young} for a nice discussion
of these issues. The Axiom A assumption ensures that the deterministic
dynamics has a rich attractor which attracts a set of positive
Lebesgue measure. \cc{need to check something here}

One area where the idea of the relevant invariant measure being
selected through a zero noise limit is prevalent is in the study of
stochastically forced and damped PDEs. Two important examples are the
stochastic Navier-Stokes equations and the stochastic KdV
equation. Both of these equations have been studied in a sequence of
works by Kuksin and his co-authors
\cite{kuksin4,kuksin5,kuksin6,kuksin2,kuksin3}. In all these works,
tightness is established by balancing the noise and dissipation as the
zero noise limit is taken. Any limiting invariant measure is shown to
satisfy an appropriate limiting equation. Typically a number of
properties are inherited from the pre-limiting invariant measure.

The hope is that the study of these limiting measures will give some
insight into important questions for the original, unperturbed
equations. In the case of the Navier-Stokes equations one would be
interested in understanding questions such as the existence of energy
cascades and turbulence. Setting aside the question of whether the
regularity of the solutions in \cite{kuksin4} is appropriate for
turbulence, it is interesting to understand if the noise selects a
unique limit and what are the obstructions to such uniqueness as they
give information about the structure of the deterministic phase
space. In all of the works
\cite{kuksin4,kuksin5,kuksin6,kuksin2,kuksin3} the question of
uniqueness of the limit is not addressed and seems out of reach.  (Though a rencent work \cite{Kuksin13} makes progress in this direction. See the
note at the end of the introduction.)

The equation for the evolution
a 2D incompressible fluid's vorticity $q(x,t)$ (a scalar) on the
2-torus subject to stochastic excitation can be written as
\begin{equation*}
  \dot  q(x,t) = \nu \Delta q(x,t)  + B\big(q(x,t),q(x,t)\big)+ \sqrt{\nu}\sum_{k \in \Z^2}
 \sigma_k e^{ik \cdot x} \dot W_t^{(k)}
\end{equation*}
where $\nu>0$ is the viscosity, $\Delta$ is the Laplacian, $\sigma_k$
are constants to be chosen, $\{W_t^{(k)} : k
\in \Z^2\}$ are a collection of  standard one-dimensional Wiener
processes and  $B(q,q)$
is a quadratic non-linearity such that $\langle B(q,q), q
\rangle_{L^2}=0$. The scaling of $\nu$ is chosen to keep the spatial
$L^2$ norm of order one in the $\nu \rightarrow 0$ limit and is the only
scaling on a fixed torus which will result in a non-trivial sequence of
tight processes. On a fixed interval, the formal $\nu =0$ limit of
this is equation is the Euler equation which conserves its
Hamiltonian (the energy or $L^2$ norm) but also has an infinite
collection of other conserved quantities since the vorticity is simply
transported about space. This means that \textit{a priori} there will
be many conserved quantities whose slow evolution must be analyzed.

Inspired by models in
\cite{Lorenz_1963,Milewski_Tabak_Vanden-Eijnden_2002} and the Euler
equation itself, we construct a model problem in the form of an ODE in
$\R^3$ such that the non-linearity is quadratic and conserves the norm
of the solution as in analogy with the Euler non-linearity. We will
also see that our
model system in fact possesses two conserved quantities (the most it
could have without becoming trivial). In many ways our analysis follows the
familiar pattern of \cite{FreidlinWentzell} in that we change time to
consider the evolution of the conserved quantities from the unforced
system on a long time interval which grows as the noise is taken to
zero. This produces a limiting system which captures the effect of the
noise. However multiple conserved quantities
are not usually treated in Wentzell-Freidlin theory and  the complications of having
more then one are non-trivial in our case. In particular, the limiting system does not have a unique solution.
 Nonetheless, we are able to show that a particular solution is selected by
the limiting procedure which in turn leads to a unique invariant
measure for the limiting system being selected.

Since the limiting system does not have unique solutions there are
many possible invariant measures depending on which of the
solutions are chosen. It is interesting to note that identifying the
limiting ``averaged'' solution is not sufficient to identify the
likely limit of the invariant measure. Analysis of the limiting
solution in isolation revel domain walls which separate different
regions of phase space and along which the diffusion degenerates
giving rise to the possibility of solutions which could spend arbitrary
mounts of time on the domain boundaries. Only through the analysis of
pre-limiting systems do we discover that the systems selects the
solutions which spend zero time on the domain walls.

These domain walls are the planes $\{x=y\}$ in $\R^3$ and correspond
to heteroclinic cycles from the original deterministic system made up
of homoclinic orbits connecting the fix points. Hence it is not
surprising that the limiting system supports solutions which could spend
arbitrarily long times on these orbits. However, it is interesting
that the limiting procedure selects solutions which do not become
trapped near the heteroclinic orbits.

After the completion and submission of this work, we became aware of a
recent work by Kuksin which proves that the zero noise/damping limit
of the stochastic Complex Ginzburg Landau (CGL) selects a unique
invariant measure from the many possible measures which are invariant
for the formal deterministic limit \cite{Kuksin13}. That problem is
very much in the spirit of the one discussed here. Our example is finite dimensional. 
However, the associated 
limiting martingale problem for the fast variables is more complicated
and does not have a unique solution. Like the CGL, our example has a
simplified orbit structure which facilitates averaging. More
complicated settings such as the stochastic Navier Stokes equations
are still out of reach. We hope our paper helps clarifying some
of the issues involved.

\section{Model System}
\label{sec:model-system}

As an exercise in studying the zero noise/dissipation limit of
conservative systems, we have chosen to study the following three--dimensional system:
\begin{align}\label{epsZero}
   \dot \xi_t =
 \BB(\xi_t,\xi_t)
\end{align}
with $\xi_0=(X_0,Y_0,Z_0) \in \R^3$
where if $\xi=(x,y,z) \in \R^3$ and $\hat \xi=( \hat x,\hat y,\hat z)$
then  $\BB$ is the symmetric bi-linear form defined by
\begin{align}
  \label{eq:B}
  \BB(\xi,\hat\xi) \eqdef  \frac12
  \begin{pmatrix}
   y\hat z + \hat y z\\
 x\hat z + \hat x z\\
-2x\hat y - 2\hat x y
  \end{pmatrix}\,.
\end{align}
We will write $\varphi_t$ for the flow map induced by
\eqref{epsZero}, i.e. $\xi_t=\varphi_t( \xi_0)$. We will constantly write
$(X_t,Y_t,Z_t)$ for $\xi_t$ when we wish to speak of the components of $\xi_t$.

Since
\begin{align}
 \label{eq:Bzero}
 B(\xi,\xi)\cdot \xi =0
\end{align}
we see that $|\xi_t|^2=X_t^2+Y_t^2+Z_t^2$ is constant along trajectories of
\eqref{epsZero}. Similarly one sees that $X_t^2-Y_t^2$ is also conserved
by the dynamics of \eqref{epsZero}. Since any linear combination is also
conserved, we are free to consider $2X_t^2+Z_t^2$ and $2Y_t^2+Z_t^2$,
which are more symmetric. Since we will typically use the
second pair, we introduce the  map
\begin{align}
  \label{eq:Phi}
  \Phi\colon (x,y,z)\mapsto (u,v)=(2x^2+z^2,2y^2+z^2)\,.
\end{align}
A moments reflection shows
that the existence of these two conserved quantities implies that all
of the orbits of \eqref{epsZero} are bounded and most are closed orbits,
topologically equivalent to a circle. All orbits live on the surface
of a sphere whose radius is dictated by the values of the conserved
quantities. More precisely, given the initial condition $\xi_0 \in
\R^3$  the orbit $\{ \xi_t : t \geq0\}$ is contained in the set
\begin{align}
  \label{eq:Orbit}
  \Gamma=\Big\{ \xi : \Phi(\xi)=\Phi(\xi_0) \Big\}
\end{align}
 To any initial point $\xi_0=(X_0,Y_0,Z_0)$ contained in a closed orbit, we can associate a measure
defined by the following limit
\begin{align}\label{ergodic}
  \mu_{\xi_0}(dx , dy , dz )\eqdef\lim_{t \rightarrow
    \infty} \frac{1}t\int_0^t \delta_{\xi_s}(dx , dy , dz )
  ds \,.
\end{align}
We will show  in Section~\ref{sec:ergod-invar-meas} that this invariant measure depends only upon
$\Phi(\xi_0)$, and a choice of a sign.
Any such  measure  is an invariant measure for the dynamics
given by \eqref{epsZero}. Hence we see that  \eqref{epsZero} has infinitely
many invariant measures. It is reasonable to expect that the addition of sufficient
driving noise and balancing dissipation, will result in a system with a
unique invariant measure. Our goal is to study its limit as the
noise/dissipation are scaled to zero. We are specifically interested
in understanding whether this procedure selects a unique convex combination  of the
measures for the underlying deterministic system \eqref{epsZero}.

 More concretely for $\epsilon>0$, we will explore the following  stochastic
differential system
\begin{equation}\label{eps}
  \begin{split}
    \dot \xi^\epsilon_t = \BB(\xi_t^\epsilon,\xi_t^\epsilon)- \epsilon
    \xi_t^\epsilon + \sqrt{\epsilon} \sigma \dot
    W_t,
 \end{split}
\end{equation}
with $\xi_0^\epsilon=(X_0,Y_0,Z_0) \in \R^3$,
$$\sigma=\begin{pmatrix}\sigma_1 & 0 & 0\\
0 & \sigma_2 & 0\\
0 & 0 & 0\end{pmatrix}
\text{ and }
W_t=(W_t^{(1)},W_t^{(2)}),$$
where the two components  $W_t^{(1)}$ and $W_t^{(2)}$ are mutually independent
standard Brownian motions. In all this paper, we assume that $\sigma_1>0$ and $\sigma_2>0$.

As above, we will write $(X^\epsilon_t,Y^\epsilon_t,Z^\epsilon_t)$ when we wish to
discuss the coordinates of $\xi^\epsilon_t$.

For each $\epsilon>0$, the three--dimensional hypoelliptic diffusion process
is positive recurrent and ergodic, its
unique invariant probability measure $\mu_\epsilon$ is absolutely
continuous with respect to Lebesgue measure, with a density which
charges all open sets.

Our aim is to study the limit of $\mu_\epsilon$, as
$\epsilon\to0$.
We first note that as $\epsilon\to0$, the process
$(X^\epsilon_t,Y^\epsilon_t,Z^\epsilon_t)$ converges to the solution of \eqref{epsZero}
on any finite time interval.

The main result of this article is that
there exists a probability  measure $\mu$ which is absolutely
continuous with respect to Lebesgue measure and
so that $\mu^\varepsilon$ converges weakly to $\mu$ as $\varepsilon
\rightarrow 0$. Of course, $\mu$ is a mixture of the ergodic invariant measures appearing in \eqref{ergodic},
which we shall describe.

It is natural to ask what is the effect of adding noise also  in the
$z$-direction. Unfortunately this leads to unexpected complications
which at present we are not able to handle.

\section{Main Results}
\label{sec:main-results}

For $\epsilon \geq 0$, we define the Markov semigroup $P_t^\epsilon$
associated with \eqref{eps} by
\begin{equation*}
  (P_t^\epsilon\phi)(\xi) = \mathbf{E}_{\xi} \phi(\xi_t^\epsilon)
\end{equation*}
for bounded $\phi\colon \R^3 \rightarrow \R$.

The dynamics obtained by formally setting $\epsilon=0$  are
deterministic.  We will write $P_t$
rather than $P^0_t$ for the corresponding Markov semigroup which is
defined by $(P_t \phi)(\xi)=\phi(\xi_t)$.

\begin{theorem}\label{invMeasureEpsPos}
  For each $\epsilon >0$, $P_t^\epsilon$ has a
  unique invariant probability measure $\mu^\epsilon$ which has a
  $C^\infty$ density which is everywhere positive.
\end{theorem}

\begin{theorem}\label{limitinvmeas}
There exists a  probability
  measure $\mu$ such that $\mu^\epsilon \Rightarrow \mu$ as
  $\epsilon \rightarrow 0$ and furthermore such that $\mu$ is invariant for the dynamics generated by
  \eqref{epsZero} in that $\mu P_t = \mu$ for all $t>0$.
   In addition,
  $\mu$ is absolutely continuous with respect to Lebesgue measure on $\R^3$, with
  a density which is positive on the complement of $\{x=y\}\cup\{x=-y\}$.
  \end{theorem}

The paper is organized as follows. Section~\ref{sec4} studies finite time convergence of $\xi^\epsilon$ as $\epsilon\to0$.
Section~\ref{sec5} studies existence and uniqueness of the invariant
measure for $\xi^\epsilon$.  Namely, this section proves
Theorem \ref{invMeasureEpsPos}. Section~\ref{sec6} studies the deterministic system on a faster time scale, more precisely it introduces the process
$(U^\epsilon_t,V^\epsilon_t)=\Phi(X^\epsilon_{t/\epsilon},Y^\epsilon_{t/\epsilon},Z^\epsilon_{t/\epsilon})$. Assuming
some results from Section~\ref{sec8}, we uniquely characterize the
limit $(U_t,V_t)$ of the process   $(U^\epsilon_t,V^\epsilon_t)$ as
$\epsilon\to0$, and show that that limit has a unique invariant
probability measure. Section~\ref{sec7} studies very precisely the
deterministic dynamics behind the ODE \eqref{epsZero} obtained by
formally setting $\epsilon=0$ in \eqref{eps}. Section~\ref{sec8}
establishes crucial results which were assumed to hold in the discussion  in Section~\ref{sec6}, the main important and most delicate one being the convergence of the quadratic variation of $(U^\epsilon,V^\epsilon)$, which builds upon the analysis in Section~\ref{sec7}. Finally Section~\ref{sec9} is devoted to the proof of Theorem \ref{limitinvmeas}.

\section{Finite Time Convergence on original timescale}\label{sec4}

In this section we show that the dynamics of  stochastic dynamics
given by \eqref{eps} converge to the deterministic dynamics given by
\eqref{epsZero}. Hence the limit as $\epsilon \rightarrow 0$ on this
time scale does not help in understanding the selection of any
limiting invariant measure as $\epsilon \rightarrow 0$.

\begin{lemma} \label{rhoBound}There exists a positive constant $c$ so that
  if $\xi_0=\xi_0^\epsilon \in \R^3$ then for all $\epsilon >0$
  \begin{align*}
    \E|\xi_t^\epsilon - \xi_t|^2 \leq \epsilon\frac{|\sigma|^2 +
      \epsilon |\xi_0|}{c |\xi_0|}  e^{c |\xi_0| t}.
  \end{align*}
\end{lemma}
\begin{corollary} For any $t \geq 0$,   $P_t^\epsilon$ converges weakly
  to $P_t$ as $\epsilon \rightarrow 0$. In other words,   for any bounded and continuous
  $\phi:\R^3 \rightarrow \R$, $P_t^\epsilon \phi(\xi) \rightarrow
  P_t\phi(\xi)$ for all $\xi \in \R^3$.
\end{corollary}

\begin{proof}[Proof of Lemma~\ref{rhoBound}]
Defining $\rho_t^\epsilon=\xi_t^\epsilon
  -\xi_t$ we have that
  \begin{align*}
    d \rho_t^\epsilon = -\epsilon \rho_t^\epsilon -\epsilon \xi_t+
    \BB(\rho_t^\epsilon,\xi_t) + \BB(\xi_t,\rho_t^\epsilon)+
    \BB(\rho^\epsilon_t,\rho_t^\epsilon) + \sqrt{\epsilon}\sigma dW_t\,.
  \end{align*}
We will make use of the following estimate which is straightforward
  to prove : there exists a $c >0$ so that
 \begin{align*}
    2|\BB(\xi,\rho)\cdot \rho| +2|\BB(\rho,\xi)\cdot \rho|+ 2 \epsilon|\xi \cdot \rho| \leq \epsilon^2 |\xi| +
  c|\xi| |\rho|^2\,.
  \end{align*}

Applying It\^o's formula to $\rho \mapsto |\rho|^2$ and this estimate produces
\begin{align*}
\frac{d\ }{dt}  \E |\rho_t^\epsilon|^2 &\leq c |\xi_t| \E
|\rho_t^\epsilon|^2 + \epsilon( |\sigma|^2 + \epsilon|\xi_t|)
\end{align*}Recalling that $|\xi_t|=|\xi_0|$ and applying Gronwall's
lemma produces the stated result.
\end{proof}

\section{Existence and Uniqueness of Invariant Measures\\ with Noise}\label{sec5}
Similarly if we consider the evolution of the norm, we have the
following result which is useful in establishing the  existence of the
invariant measure $\mu^\varepsilon$ and the tightness of various objects.

\begin{proposition}\label{prop:energy} For any integer $p\geq1$ there exists $C(p)>0$  so
  that for all $t\geq 0$, $\epsilon>0$,
  \begin{align*}
 \E |\xi_t^\epsilon|^{2p} \leq C(p) \Big[ 1+ e^{-2\epsilon t}
 \sum_{k=1}^p |\xi_0|^{2k} \Big]
  \end{align*}
\end{proposition}\label{tightXi}

\begin{proof}[Proof of Proposition~\ref{prop:energy}] Defining $|\sigma|^2=\sum_i
  \sigma_i^2$, It\^o's formula implies that
  \begin{align*}
    d |\xi_t^\epsilon|^2 = -2\epsilon |\xi_t^\epsilon|^2 dt + |\sigma|^2 dt + dM_t^\epsilon
  \end{align*}
for a martingale $M_t^\epsilon$ with  quadratic variation satisfying
\begin{align*}
  d\langle M^\epsilon \rangle_t = (\sigma_1^2 (X_t^\epsilon)^2 + \sigma_2^2 (Y_t^\epsilon)^2 ) dt \leq \sigma_{max}^2 |\xi_t^\epsilon|^2 dt
\end{align*}
where $\sigma_{max}^2 =  \sigma_1^2\vee\sigma_2^2$.
The proof then follows from Lemma~\ref{l:boundingLemma} below.
\end{proof}

\begin{remark}
  One can actually easily prove uniform in time bounds on $\E\exp(
  \kappa X_t)$ for $\kappa>0$ but sufficiently small. See \cite{MH08}
  for a  proof using the exponential martingale estimate.
\end{remark}

The following Lemma provides the key estimate to
Proposition~\ref{prop:energy}.
\begin{lemma}\label{l:boundingLemma}
  Let $X_t$ be a semimartingle so that $X_t\geq0$,
  \begin{align*}
    d X_t = (a - b X_t) dt + dM_t
  \end{align*}
where $a>0$, $b>0$ and $M_t$ is a continuous local martingale satisfying
\begin{align*}
  d\langle M\rangle_t \leq c X_t dt
\end{align*}
for some $c>0$. Then for any integer $p\geq1$ there exist a constant $C(p)$
(depending besides $p$ only on $a$, $b$ and $c$) so
that for any $X_0\geq 0$ and $t \geq 0$
\begin{align*}
  \E \big[  X_t^{p}  \big] \leq C(p) \big[ 1+  \sum_{k=1}^p e^{-bkt}  X_0^{k}   \big]
\end{align*}
\end{lemma}
\begin{proof}[Proof of Lemma~\ref{l:boundingLemma}]
  Fixing an $N >0$ and defining the  stopping time $\tau=\inf\{ t : X_t > N\}$
observe that
\begin{align*}
   \E X_{t} \leq \E X_{t \wedge \tau} \leq a t + X_0\,.
\end{align*}
where the first inequality follows from Fatou's lemma applied to the
limit $N \rightarrow \infty$. Using the assumption on the quadratic
variation of $M_t$ we now see that $M_t$ is a $L^2$-martingale. Hence
\begin{align}\label{initalBound}
  \E X_t = e^{-b t} X_0 + \frac{a}{b} (1 - e^{- b t})
\end{align}
Now applying It\^o's formula to $X_t^{p}$ produces
\begin{align*}
  d X_t^{p}= p X_t^{p-1} (a -b X_t)dt  + \frac{p(p-1)}{2}X_t^{p-2}d\langle M\rangle_t + dM_t^{(p)}
\end{align*}
where $dM_t^{(p)}= p X_t^{p-1} dM_t$.  Using the same stopping time
$\tau$ and the same argument as before, we have
\begin{align*}
\E  X_{t}^p \leq\E  X_{t\wedge \tau}^p \leq   X_0^{p}+  (p a +cp(p-1) )\int_0^t \E  X_s^{p-1}ds
\end{align*}
Hence inductively we have a bound on $\E X_t^p$ for all integer $p\geq 1$
which implies that $M_t^{(p)}$ is an $L^2$-martingale for all $p\geq
1$. Hence we have
\begin{align*}
  \E X_t^p \leq e^{-bpt} X_0^p +  (p a +\frac{c}{2}p(p-1) )\int_0^t e^{-bp (t-s)}\E  X_s^{p-1}ds \,.
\end{align*}
Proceeding inductively using this estimate and \eqref{initalBound} as
the base case produces the stated result.
\end{proof}

\begin{corollary}\label{cor:xiInvMeasure}
  For each $\epsilon > 0$, the Feller diffusion $\{\xi^\epsilon_t: t
  \geq 0\}$ possesses at least one invariant probability measure
  $\mu^\epsilon$. Furthermore, any invariant probability measure $\mu^\epsilon$  satisfies
  \begin{align*}
    \int_{\R^3} |\xi|^p \mu^\epsilon(d\xi) \leq C(p)
  \end{align*}
for any integer $p \geq1$ where $C(p)$ is the constant from
Lemma~\ref{prop:energy} (which is independent of $\epsilon$). Hence
the collection of probability measures which are invariant  under the
dynamics for any given
$\epsilon  >0$ is tight.
\end{corollary}
\begin{proof}[Proof of Corollary~\ref{cor:xiInvMeasure}]
   Since $\xi^\epsilon_t$ is a   time--homogeneous Feller diffusion
process and from Proposition \ref{prop:energy} for fixed $\epsilon>0$,
the collection of random vectors $\{\xi^\epsilon_t,\, t>0\}$ is tight, the
existence of an invariant probability measure $\mu^\epsilon$ follows by the
Krylov--Bogolyubov theorem.

Defining $\phi_{N,p}(\xi)=|\xi|^{2p}\phi(|\xi|/N)$ where $\phi$ is a
smooth function such that $\phi(x)=1$ for $x\leq1$, $\phi(x)=0$ for $x
\geq2$, and  $\phi$ decreases
monotonically on $(1,2)$, we see that
\begin{align*}
 \int \phi_{N,p}(\xi) \mu^\epsilon(d \xi) &=\int \E_{\xi_0} \phi_{N,p}(\xi_t^\epsilon) \mu^\epsilon(d \xi_0)\\
  &\leq C(p) \big[ 1+ e^{-2\epsilon t}  \int \phi_{N,p}(\xi)
  \mu^\epsilon(d \xi) \big]\\
  &\leq C(p) \big[ 1+ e^{-2\epsilon t} (N+1)^{2p} \big]
\end{align*}
Taking $t \rightarrow \infty$, followed by $N\rightarrow \infty$, the
result follows from Fatou's Lemma. Since these bounds are uniform in
$\epsilon$,  tightness follows immediately.
\end{proof}

The next result follows by hypoellipticity and the Stroock and Varadhan
support theorem.
\begin{proposition}\label{prop:posDensity}
For any $\e>0$, there exists a transition
  density  $p^\e_t(\xi,\eta)$ which is jointly smooth in
  $(t,\xi,\eta)$ so that for all $\xi\in\R^3$ and Borel $A \subset
  \R^3$ one has
  \begin{align*}
    P^\epsilon_t(\xi,A) = \int_A p_t^\epsilon(\xi,\eta) \,d\eta\,.
  \end{align*}
Additionally, $\int_Bp_t^\epsilon(\xi,\eta)d\eta>0$ for every $\epsilon>0$, $t>0$,
 $\xi \in \R^3$, and any ball $B\subset\R^3$.
\end{proposition}

\begin{proof}[Proof of Proposition~\ref{prop:posDensity} ]

  Hypoellipticity follows from the
  fact that taking Lie brackets of the drift with $\partial_x$ and
  then $\partial_y$ (the two noise directions) produces the third and
  missing direction $\partial_z$. This ensures the
  existence of a smooth density with respect to Lebesgue
  mesure \cite{Stroock08,hormander94III,hormander94IV}. Positivity will
  then follow by showing that the support of the transition density is
  all of $\R^3$.

We will invoke the support theorem of Stroock and Varadhan \cite{support}.
Indeed, consider the controlled system associated to the SDE for
$(X^\e_t,Y^\e_t,Z^\e_t)$, which reads
\begin{equation}
\begin{aligned}
  \frac{dx^\e}{dt}(t)&=y^\e(t)z^\e(t)-\e x^\e(t)+\sqrt{\e}\sigma_1 f_1(t)\\
  \frac{dy^\e}{dt}(t)&=x^\e(t)z^\e(t)-\e y^\e(t)+\sqrt{\e}\sigma_2 f_2(t)\\
  \frac{dz^\e}{dt}(t)&=-2x^\e(t)y^\e(t)-\e z^\e(t),
\end{aligned}
\end{equation}
where $\{(f_1(t),f_2(t)),\ t\ge0\}$ is the control at our disposal. Now by
choosing appropriately the control, we can drive the two components
$(x^\e(t),y^\e(t))$ in time as short as we like to any desired
position, which permits us to drive the last component $z^\e(t)$ to
any prescribed position in any prescribed time. The result follows.
\end{proof}

We are now in a position to give the proof of
Theorem~\ref{invMeasureEpsPos}.
\begin{proof}[Proof of Theorem~\ref{invMeasureEpsPos}]
  Since by Proposition~\ref{prop:posDensity}, $P_t^\epsilon$ has a
  smooth transition density, any invariant probability measure must have a smooth
  density which charges any ball $B\subset\R^3$. Recall the fact that
  in our setting any two distinct ergodic invariant probability measures must have
  disjoint supports which is impossible since the measures have
  densities and charge any open set.  Uniqueness of the invariant probability measure
   follows immediately from
  the fact that any invariant measure can be decomposed into ergodic
  components \cite{Sinai87}.
\end{proof}

\section{The fast dynamics}\label{sec6}
\label{sec:fast-dynamics}
Since by the results in Section~\ref{sec4} $\xi^\epsilon_t=(X^\varepsilon_t,Y^\varepsilon_t,Z_t^\varepsilon)$ converges to
$\xi_t=(X_t,Y_t,Z_t)$, in order to study the limiting invariant probability measure one
needs to consider the system on ever increasing time intervals as
$\varepsilon \rightarrow 0$. One must pick a time scale, depending
on $\epsilon$, so that the amount  of randomness injected into
the system is sufficient to keep  the system from settling onto a
deterministic trajectory as $\epsilon \rightarrow 0$.

With this in mind consider the process $\xi_t^\epsilon$
on the fast scale $t/\epsilon$. In other words, consider the process
$\tilde \xi_t^\epsilon= \xi_{t/\epsilon}^\epsilon$ which
 solves the
SDE
\begin{equation}\label{fast}
   \begin{split}
    \dot{\widetilde{\xi}^\epsilon_t} =\frac1\epsilon \BB(\widetilde{\xi}_t^\epsilon,\widetilde{\xi}_t^\epsilon)-  \widetilde{\xi}^\epsilon +\sigma \dot
    W_t\,.
 \end{split}
\end{equation}
 Here we have used a slight abuse of notations, replacing the
$\epsilon$--dependent standard Brownian motion $W^\epsilon_t=
\sqrt{\epsilon}W_{t/\epsilon}$ by $W_t$.
In coordinates we will write
$(\widetilde{X}^\epsilon_t,\widetilde{Y}^\epsilon_t,\widetilde{Z}^\epsilon_t)=(X^\epsilon_{t/\epsilon},Y^\epsilon_{t/\epsilon},Z^\epsilon_{t/\epsilon})$.

Let $\widetilde P^\varepsilon_t$ be the Markov semigroup associated to \eqref{fast} and
defined for $\psi:\R^3 \rightarrow\R$  by
\begin{align}
  \label{eq:Pt}
  (\widetilde P^\varepsilon_t\psi)(\xi)=\E_{\xi} \psi(\widetilde
  \xi_t^\epsilon)\,.
\end{align}
Associated with this right-action on functions we have a  dual
action on measures. We will denote this by left action rather then the
often used $(\widetilde P_t^\epsilon)^*$ notation. Hence if $\mu$ is a measure on
$\R^3$ and $\psi$ a real-valued function on $\R^3$ then
\begin{align*}
  \mu \widetilde P_t^\epsilon\psi = \int_{\R^3} (\widetilde P_t^\epsilon \psi)(\xi)\, \mu(d\xi)\,.
\end{align*}

Of course, this time change does not change the set of invariant
measures for the dynamics since the time change was not state dependent. Hence  $\widetilde P_t^\epsilon$
has the same unique invariant probability measure $\mu^\epsilon$ as $P_t^\epsilon$.

\subsection{Fast evolution of conserved quantities}
One indication that this is the right time scale is that the conserved
quantities $(u,v)=\Phi(x,y,z)$ now continue to evolve randomly as
$\varepsilon \rightarrow 0$. More precisely, defining the processes
$(U^\epsilon_t$ $V^\epsilon_t)$ by
$(U_t^\varepsilon,V_t^\varepsilon)=\Phi(\widetilde{X}_t^\varepsilon,\widetilde{Y}_t^\varepsilon,\widetilde{Z}_t^\varepsilon)$
and applying
 It\^o's formula shows that
\begin{equation}\label{rho-chi}
  \begin{split}
    dU^\epsilon_t&=[2\sigma_1^2-2U^\epsilon_t]dt+4\sigma_1\widetilde{X}^\epsilon_tdW^{(1)}_t,\\
    dV^\epsilon_t&=[2\sigma_2^2-2V^\epsilon_t]dt+4\sigma_2\widetilde{Y}^\epsilon_tdW^{(2)}_t.
  \end{split}
\end{equation}
We will show below in Section~\ref{sec8} that
$\{(U_t^\varepsilon,V_t^\varepsilon)\}_{\epsilon>0}$ is tight, and that any accumulation point
$(U_t,V_t)$  solves the SDE
\begin{equation}\label{UV}
  \begin{split}
    dU_t&=[2\sigma_1^2-2U_t]dt +\sigma_1 \sqrt{8\big(U_t - \Gamma(U_t,V_t)  \big)}dW^{(1)}_t,\\
    dV_t&=[2\sigma_2^2-2V_t]dt +\sigma_2\sqrt{8\big(V_t - \Gamma(U_t,V_t) \big)}dW^{(2)}_t,
  \end{split}
\end{equation}
where
\begin{align}\label{LambdaToGamma}
  \Gamma(u,v)=(u\wedge v)\Lambda\left(\frac{u\wedge v}{u\vee
      v}\right).
\end{align}
The function $\Lambda$ will be defined in Section~\ref{sec:Lambda}. However for
our present discussion, it will be sufficient to state a few important
facts whose proofs will be given later in Section~\ref{sec:Lambda}.
\begin{proposition} \label{prop:Lambda}
$\Lambda(r)$ is a continuous and strictly increasing
function on
$[0,1]$ with $\Lambda(0)=\frac12$ and
$\Lambda(1)=1$. Furthermore as $\varepsilon \rightarrow 0^+$,
\begin{align*}
\Lambda(\varepsilon)&=\frac12 + \frac{1}{16}\epsilon  + \frac{1}{32}\epsilon^2 + o(\epsilon^2),\\
\Lambda(1-\varepsilon) &= 1 - \frac{2}{|\ln(\varepsilon)|} + o\Big(
\frac{1}{|\ln(\varepsilon)|} \Big).
\end{align*}
In addition, on any closed interval in $[0,1)$, $\Lambda$ is uniformly Lipschitz.
\end{proposition}
\subsection{Finite time behavior $(U_t,V_t)$}
\label{sec:finite-time-behavior}
Before stating and proving the main theorem of this section, let us
establish three Lemmata.
\begin{lemma}\label{le:posit0}
  Let $\{X_t,\, t\ge0\}$ be a continuous $\R_+$--valued
  $\F_t$--adapted process which satisfies
\begin{align*}
dX_t&=(a-b X_t)dt+\sqrt{c X_t}dW_t,\\
X_0&=x,
\end{align*}
where $b, c >0$, $\{W_t,\, t\ge0\}$ is a standard Brownian adapted to  $\F_t$.
motion and $x>0$. If $a\ge c/2$, then with probability one $X_t>0$ for all $t\ge0$.
\end{lemma}
\begin{proof}[Proof of Lemma~\ref{le:posit0}]
  We consider the SDE
\begin{align*}
dY_t&=c^{-1}\left(aY_t-b Y^2_t\right)dt+Y_tdW_t,\\
Y_0&=x,
\end{align*}
whose solution satisfies clearly
\begin{align*}
  Y_t=x\exp\left(\left[\frac{a}{c}-\frac{1}{2}\right]t-\frac{b}{c}\int_0^tY_sds+W_t\right)\,,
\end{align*}
and define $A_t=\frac{1}{c}\int_0^t Y_sds$, $\eta_t=\inf\{s>0,\,
A_s>t\}$ and $X_t=Y_{\eta_t}$. It is not hard to see that there exists
a standard Brownian motion, which by an abuse of notation we denote
again by $W$, which is such that
\begin{align*}
dX_t&=(a-b X_t)dt+\sqrt{c X_t}dW_t,\\
X_0&=x.
\end{align*}
Since $X_{A_t}=Y_t$ and $Y_t>0$ for all $t<\infty$, for $X_t$ to hit
zero in finite time, it is necessary that $A_\infty<\infty$ and
$Y_\infty=0$.  But from the above formula for $Y_t$, we deduce that if
$a\ge c/2$, on
the event $\{A_\infty<\infty\}$, $\limsup_{t\to\infty}Y_t=+\infty$,
which implies that $A_\infty=+\infty$.
\end{proof}

We will need a slightly better result which generalizes the preceding Lemma.
\begin{lemma}\label{le:posit}
Let $\{X_t,\, t\ge0\}$ and $\{Y_t,\, t\ge0\}$ be continuous $\R_+$--valued $\F_t$--adapted
processes which satisfy $0\le Y_t\le X_t$ for all $t\ge0$, with $Y_0>0$,
\begin{align*}
dX_t&=(a-b X_t)dt+\sqrt{c Y_t}dW_t,\\
X_0&=x,
\end{align*}
where $b, c >0$, $\{W_t,\, t\ge0\}$ is a standard $\F_t$--Brownian
motion and $x>0$. If $a\ge c/2$, then a. s. $X_t>0$ for all $t\ge0$.
\end{lemma}
\begin{proof}[Proof of Lemma~\ref{le:posit}]
We define
\begin{align*}
B_t&=\int_0^t \frac{Y_s}{X_s}ds,\quad
\sigma_t=\inf\{s>0,\, B_s>t\},\quad\text{and}\quad
Z_t=X_{\sigma_t}.
\end{align*}
There exists a standard Brownian motion, still denoted by $W$, such that
\begin{align*}
dZ_t&=(a-b Z_t)\frac{Z_t}{Y_{\sigma_t}}dt+\sqrt{c Z_t}dW_t,\\
X_0&=x.
\end{align*}
Define two sequences of stopping times as follows. $S_0=0$, and for $k\ge1$,
\begin{align*}
T_k&=\inf\left\{t>S_{k-1},\, Z_t<\frac{a}{2b}\right\}\quad\text{and}\quad
S_k=\inf\left\{t>T_{k},\, Z_t>\frac{a}{b}\right\}.
\end{align*}
On each interval $[T_k,S_k]$, since $Z_t/Y_{\sigma_t}\ge1$, by a
standard comparison theorem for SDEs we can bound from below $Z_t$ by
the solution of the equation of the previous Lemma, starting from
$a/2b$. Hence $Z_t$ never hits zero.
\end{proof}

Using Lemma~\ref{le:posit}, we are now in a position to prove the
following result.
\begin{theorem}\label{th:uniq}
Assume that the initial condition $(U_0,V_0)$ satisfies $U_0 >0$, $V_0>0$.
Any solution of
equation \eqref{UV}  lives in the
set
$(0,+\infty)\times(0,+\infty)$  for all positive times.
\end{theorem}
\begin{proof}
 We first prove that any solution $(U_t,V_t)$
never hits the two axis, except possibly at $(0,0)$.  The fact that $(U_t,V_t)$
cannot hit $(0,v)$ with $v>0$ follows clearly from the equation for
$U_t$ and Lemma \ref{le:posit}, once we have noted that whenever
$U_t<V_t$, $U_t-\Gamma(U_t,V_t)\le U_t/2$, as follows from Proposition
\ref{prop:Lambda}.

The same proof shows that $(U_t,V_t)$ cannot hit $(u,0)$ with
$u>0$. It remains to show that $(U_t,V_t)$ cannot hit $(0,0)$.
 Let $a=\sigma_1^{-2}$, $b=\sigma_2^{-2}$,
$K_t=aU_t+bV_t$. There exists a standard Brownian motion $W_t$ such
that
\begin{align*}
  dK_t=(4-2K_t)dt+ \sqrt{8[K_t-(a+b)\Gamma(U_t,V_t)]}dW_t.
\end{align*}
The result again follows from Lemma \ref{le:posit}, since $\Gamma(U_t,V_t)\ge0$.
\end{proof}

We next establish a crucial property shared by all the possible accumulation points of the collection $\{(U^\epsilon,V^\epsilon)\}_{\epsilon>0}$.

\begin{theorem}\label{accumul}
If $(U,V)=\lim_{n\to\infty}(U^{\epsilon_n},V^{\epsilon_n})$, for some sequence  $\epsilon_n\to0$, then
\begin{equation*}
  \int_0^t{\bf1}_{\{U_s=V_s\}}ds=0\text{  for all  }t>0 \text{ almost surely}\,.
\end{equation*}
\end{theorem}
\begin{proof}
For any $M, \delta>0$, let
\begin{align*}
  \psi_{M,\delta}(x)=\left[\log\left(\frac{1}{|x|}\right)\wedge M\right]{\bf1}_{[-\delta,\delta]}(x), \
  \psi_\delta(x)=\log\left(\frac{1}{|x|}\right){\bf1}_{[-\delta,\delta]}(x).
\end{align*}
We define $\varphi_{M,\delta}$ and $\varphi_\delta\in C^1(\R)\cap W^{2,\infty}(\R)$ by
\begin{align*}
\varphi_{M,\delta}(0)=0,&\ \varphi_{M,\delta}'(0)=0, \varphi''_{M,\delta}=\psi_{M,\delta},\\
\varphi_{\delta}(0)=0,&\ \varphi_{\delta}'(0)=0, \varphi''_{\delta}=\psi_{\delta}.
\end{align*}
Let $J^\epsilon_t=U^\epsilon_t-V^\epsilon_t$. It follows from It\^o's formula, which can be applied here although $³\varphi_{M,\delta}\not\in C^2(\R)$,  that for each $M>0$, $n\ge1$,
\begin{multline*}
8\E\int_0^t \left[\sigma_1^2(\tilde{X}^{\epsilon_n}_s)^2+\sigma_2^2(\tilde{Y}^{\epsilon_n}_s)^2\right]
\psi_{M,\delta}(J^{\epsilon_n}_s)ds\\
\le\E\left(\varphi_{M,\delta}(J^{\epsilon_n}_t)-\varphi_{M,\delta}(J^{\epsilon_n}_0)-2\int_0^t(\sigma_1^2-\sigma_2^2-J_s^{\epsilon_n})\varphi'_{M,\delta}(J^{\epsilon_n}_s)ds\right).
\end{multline*}
We now let $M\to\infty$. Since $\sigma_1^2(\tilde{X}^{\epsilon_n}_s)^2+\sigma_2^2(\tilde{Y}^{\epsilon_n}_s)^2>0$
$(ds\times d\P)$-almost everywhere, we deduce that
\begin{multline*}
8\E\int_0^t \left[\sigma_1^2(\tilde{X}^{\epsilon_n}_s)^2+\sigma_2^2(\tilde{Y}^{\epsilon_n}_s)^2\right]
\psi_\delta(J^{\epsilon_n}_s)ds\\
\le\E\left(\varphi_{\delta}(J^{\epsilon_n}_t)-\varphi_{\delta}(J^{\epsilon_n}_0)-2\int_0^t(\sigma_1^2-\sigma_2^2-J_s^{\epsilon_n})\varphi'_{\delta}(J^{\epsilon_n}_s)ds\right).
\end{multline*}
We now take the limit in the last inequality as $n\to\infty$, and deduce from Fatou's Lemma that
\begin{multline*}
   \E\int_0^t \left[\sigma_1^2(U_s-\Gamma(U_s,V_s))+\sigma_2^2(V_s-\Gamma(U_s,V_s))\right]
\psi_\delta(J_s)ds\\
\le\E\left(\varphi_{\delta}(J_t)-\varphi_{\delta}(J_0)-2\int_0^t(\sigma_1^2-\sigma_2^2-J_s)\varphi'_{\delta}(J_s)ds\right).
\end{multline*}
It follows from Proposition \ref{prop:Lambda} that
to any $c>0$, we can associate $\delta>0$ and $a>0$ such that whenever $u,v\ge c>0$, and $-\delta\le k=u-v\le \delta$,
$$4\left[\sigma_1^2(u-\Gamma(u,v))+\sigma_2^2(v-\Gamma(u,v))\right]
\log\left(\frac{1}{|k|}\right)\ge a>0.$$ Consequently the above establishes that
\begin{multline*}
  a\E\int_0^t{\bf1}_{U_s\ge c,V_s\ge c}{\bf1}_{(-\delta,\delta)}(J_s)ds\\\le\E\left(\varphi_{\delta}(J_t)-\varphi_{\delta}(J_0)-2\int_0^t(\sigma_1^2-\sigma_2^2-J_s)\varphi'_{\delta}(J_s)ds\right),
\end{multline*}
and letting finally $\delta\to0$, we deduce that
\begin{equation*}
  \E\int_0^t{\bf1}_{U_s\ge c,V_s\ge c}{\bf1}_{\{0\}}(J_s)ds=0.
\end{equation*}
Since we know that both $U_t$ and $V_t$ never reaches $0$, and $c>0$ is arbitrary, this shows that $J_t=U_t-V_t$ spends a.s. zero time at 0, i.e. that the process $(U_t,V_t)$ spends a.s. zero time on the diagonal.
\end{proof}

Now that we know that $(U,V)$ spends no time on the diagonal,
we can introduce the following time change. Let for $u,v>0$,
\begin{align}\label{eqF}
F(u,v)&=\begin{cases}1-\Lambda\left(\frac{u\wedge v}{u\vee v}\right),&\text{if
$\frac{u\wedge v}{u\vee v}\ge\frac{1}{2}$,}\\
1-\Lambda\left(\frac{1}{2}\right),&\text{if $\frac{u\wedge v}{u\vee v}<\frac{1}{2}$.}\end{cases}
\end{align}
Let us define the time change
 \begin{equation}\label{tc}
 \begin{split}
 A_t&=\int_0^t F(U_s,V_s) ds,\ \eta_t=\inf\{s>0,\,A_s>t\},\\
 H_t&=U_{\eta_t},\ \text{  and }K_t=V_{\eta_t}.
\end{split}
\end{equation}
There exists a two--dimensional Wiener process, which we still denote by
$(W^{(1)}_t,W^{(2)}_t)$, such that
\begin{equation}\label{HKeq}
\begin{split}
dH_t&=2\frac{\sigma^2_1-H_t}{F(H_t,K_t)}dt
+2\sqrt{2}\sigma_1\sqrt{\frac{H_t-H_t\wedge K_t}{F(H_t,K_t)}+(H_t\wedge K_t)G(H_t,K_t)}\
dW^{(1)}_t,\\
dK_t&=2\frac{\sigma^2_2-K_t}{F(H_t,K_t)}dt+
2\sqrt{2}\sigma_2\sqrt{\frac{K_t-H_t\wedge K_t}{F(H_t,K_t)}+(H_t\wedge K_t)G(H_t,K_t)}\
dW^{(2)}_t,
\end{split}
\end{equation}
where
\begin{equation*}
  G(H_t,K_t)=\frac{1-\Lambda\left(\frac{H_t\wedge K_t}{H_t\vee K_t}\right)}{F(H_t,K_t)}\,.
\end{equation*}
It is easily verified that the diffusion matrix of this system is locally uniformly elliptic in
$(0,+\infty)\times(0,+\infty)$ and continuous.
However, the drift is unbounded near the diagonal.
We will now prove uniqueness of the weak solution of \eqref{HKeq}, using methods and results from Portenko \cite{Por}. \cite{Por} constructs a weak solution to an equation like \eqref{HKeq} from the solution without drift, using Girsanov's theorem, provided the drift is in $L^p(\R^2)$, with $p>4$. His uniqueness theorem is proved under conditions which are difficult to verify. The condition is tailored to make sure that Girsanov's theorem can be used
to show that the law of the equation with drift is absolutely continuous with respect to that of the equation without drift.
We will do that by verifying the condition of Lemma 1.1 from \cite{Por}, which we now state
\begin{lemma}\label{Por1.1}
Assume that $\{Z_t,\ t\ge0\}$ is a non--negative progressively measurable process, adapted to the $\sigma$--algebra
$\{\F_t,\ t\ge0\}$. Suppose that there exists a mapping $\rho$ from the set of subintervals of $[0,T]$ into
$\R_+$, such that
\begin{enumerate}
\item[(i)] $\E\left[\int_s^tZ_rdr\Big|\F_s\right]\le \rho(s,t)$, for all $0\le s<t\le T$;
\item[(ii)] $\rho(t_1,t_2)\le\rho(t_3,t_4)$, whenever $(t_1,t_2)\subset(t_3,t_4)$;
\item[(iii)]
  \begin{align*}
    \lim_{h\downarrow0}\ \sup_{\substack{0\le s<t\le T\\ t-s\le h}}\ \rho(s,t)=\kappa
  \end{align*}
\end{enumerate}
Then for any $\lambda <\kappa^{-1}$,
\begin{align*}
  \E\exp\left(\lambda\int_0^TZ_tdt\right)<\infty.
\end{align*}
\end{lemma}
We intend to apply Proposition~\ref{Por1.1}  for a $Z_t$ which will
give us sufficient control over the drift in \eqref{HKeq} that we can
use Novikov's criterion  and Girsanov's Theorem to transform the SDE \eqref{HKeq} into the same equation without
drift, and hence prove uniqueness of the solution using an argument  in the vein of Theorem 1.2 from
\cite{Por}.

To better understand how to use Lemma~\ref{Por1.1} to remove the drift
from \eqref{HKeq} en route to prove uniqueness of the solution, we
take a close look at the drift term.
 Our uniqueness  argument exploits the fact that, since uniqueness is
a local property, we can modify the coefficients of the $(H,K)$
equation outside the set $\{(h,k),\ M^{-1}\le h, k\le M\}$ for some
arbitrary $M>0$, so that tr resulting equation takes the form
\begin{equation}\label{HKSimple}
  \begin{aligned}
dH_t&=\frac{f_1(H_t)}{F(H_t,K_t)}dt+\sigma_1\sqrt{g(H_t,K_t)}dW^{(1)}_t,\\
dK_t&=\frac{f_2(K_t)}{F(H_t,K_t)}dt+\sigma_2\sqrt{g(K_t,H_t)}dW^{(2)}_t.
  \end{aligned}
\end{equation}
where for some $C>0$, $C^{-1}\le g(h,k)\le C$ and
$|f_1(h)|+|f_2(k)|\le C$, for all $h,k\in\R$. Hence the only possible
difficulty will arises if  when $F$ is small. From the definition of $F$ in \eqref{eqF} and
the asymptotics in Lemma~\ref{prop:Lambda}  we note that $F(h,k)$ is
small when $\frac{|h-k|}{h\vee k}$ is small.

Hence if we wish to use Girsanov's theorem to remove the drift from
\eqref{HKSimple}, the danger comes from $|h-k|$ small if we restrict
our attention to the set $\{(h,k),\ M^{-1}\le h, k\le M\}$. The
following simple observation is useful to control the drift.
\begin{lemma}\label{tric}
If $a, b>0$ then
\begin{align*}
  \log\left(\frac{b}{a}\right)\le b +\log\left(\frac{1}{a}\right){\bf1}_{\{a\le1\}}= b +|\log\left(a\right)|{\bf1}_{\{a\le1\}}.
\end{align*}
\end{lemma}
\begin{proof}
  The inequality follow from the fact that if $b >0$ then $b  > \log(b)$
  and if $a\in (0,1]$ then $\log(1/a) \ge 0$. The equality follows
  from $\log\left(\frac{1}{a}\right){\bf1}_{\{a\le1\}}=|\log\left(\frac{1}{a}\right)|{\bf1}_{\{a\le1\}}=|\log\left(a\right)|{\bf1}_{\{a\le1\}}$.
\end{proof}
Combining Lemma~\ref{tric} and
the asymptotics in Lemma~\ref{prop:Lambda}  we note that
\begin{align*}
  \frac{2}{F(h,k)}\lesssim|h\vee k|+ |\log(|h-k|)|{\bf 1}_{|h-k| \leq
    1}.
\end{align*}
Hence on the set  $M^{-1}\le h\vee k\le M$, controlling $F^{-1}(h,k)$ amounts to estimating $|\log(|h-k|)|$
for $|h-k|<a$, where $0<a\le1$ is arbitrary.

  Hence, defining $N_t:=H_t-K_t$, if we desire to apply Lemma~\ref{Por1.1} to $Z_t=1/F^2(H_t,K_t)$
  it will be sufficient to estimate
\begin{multline}
  \label{eq:neededEstimate}
  \E\left[\int_s^t{\bf1}_{\{|N_r|\le1/2\}}\log^2 |N_r|\, dr\Big|\F_s\right]=\\
\E\left[\int_s^t{\bf1}_{\{|N_r|\le1/2\}}\log^2|N_r|\,dr\Big|H_s,K_s\right],
\end{multline}
for $0\le H_s,K_s\le M$ for some $M>0$. Notice that as we will
eventually prove that Lemma~\ref{Por1.1} holds with $\kappa=0$ the
conclusion of Lemma~\ref{Por1.1} will be more than sufficient to
invoke Novikov's criterion.

Recall that $N_t$ satisfies
\begin{align*}
  dN_t=\frac{f_1(H_t)-f_2(K_t)}{F(H_t,K_t)}dt+\sqrt{\overline{g}(H_t,K_t)}dW_t
\end{align*}
with $f_1$, $f_2$, $\overline{g}$ bounded and $\overline{g}$ bounded
away from zero. As a prologue to the needed estimate on
\eqref{eq:neededEstimate}, we prove the following result.
\begin{lemma}\label{timespent}
Let $0<a<b<2a\le1$. For any $M>0$, there exists a constant $C_M$ such that for any $0\le s<t$ with $t-s\le1$,
$|H_s|,|K_s|\le M$,
$$\E\left[\int_s^t{\bf1}_{\{a\le |N_r|\le b\}}\Big|H_s,K_s\right]\le \left[(t-s)\wedge C_M(b-a)^{4/5}\right].$$
\end{lemma}
\bpf
It clearly suffices to prove that
$$\E\left[\int_s^t{\bf1}_{\{a\le |N_r|\le b\}}\Big|H_s,K_s\right]\le C_M(b-a)^{4/5}.$$
We prove this result with $\{a\le |N_r|\le b\}$ replaced by $\{a\le N_r\le b\}$. The same proof would estimate similarly $\{-b\le N_r\le -a\}$.
Let $\varphi_{a,b}\in C^1(\R)\cap W^{2,\infty}(\R)$ be defined by $\varphi_{a,b}(0)=\varphi'_{a,b}(0)=0$,
$\varphi''_{a,b}(x)={\bf1}_{[a,b]}(x)$. Also $\varphi_{a,b}\not\in C^2$, we can apply It\^o's formula
to obtain
\begin{align*}
\frac{1}{2}\int_s^t\overline{g}^2(H_r,K_r){\bf1}_{\{a\le N_r\le b\}}dr&=\varphi_{a,b}(N_t)-\varphi_{a,b}(N_s)\\
&\qquad -\int_s^t\frac{f_1(H_r)-f_2(K_r)}{F(H_r,K_r)}\varphi'_{a,b}(N_r)dr\\
&\le (b-a)[|N_t|+|N_s|]\\&\qquad+C(b-a)\int_s^t \left[1+\log(|N_r|^{-1}){\bf1}_{\{a\le N_r\le 1\}}\right]dr\\
&\le C[1+|N_t|+|N_s|+\log(1/a)](b-a).
\end{align*}
Since there exists $C_M$ such that $\E[|N_t|+|N_s|\  |\  H_s,K_s]\le C_M$ for $|H_s|\le M$, $|K_s|\le M$ and $t-s\le 1$,  and $\sup_{0<a\le1}\log(1/a)(b-a)^{1/5}<\infty $ (we recall that $b-a\le a$), the result is proved.
\epf

With this result in hand, we now return to the needed estimate on \eqref{eq:neededEstimate}
which is contained in the following Lemma.
\begin{lemma}\label{port}
For any $M>0$, there exists a constant $C_M$ such that, for any $0\le s<t$ with $t-s\le1$,
on the event $|H_s|, |K_s|\le M$,
$$ \E\left[\int_s^t{\bf1}_{\{|N_r|\le1/2\}}\log^2 |N_r|\, dr\Big| H_s,K_s\right] \le C_M (t-s)^{1/4}.$$
\end{lemma}
\bpf
In this proof, $C$ stands for a constant which may vary on each line.
\begin{align*}
\E&\left[\int_s^t{\bf1}_{\{|N_r|\le1/2\}}\log^2|N_r|\,dr\Big|H_s,K_s\right]\\ &\qquad=
\sum_{k=2}^\infty\E\left[\int_s^t{\bf1}_{\{(k+1)^{-1}\le |N_r|\le
    k^{-1}\}}\log^2 |N_r|\, dr\Big|H_s,K_s\right]\\
&\qquad\le\sum_{k=2}^\infty\log^2(k+1)\E\left[\int_s^t{\bf1}_{\{(k+1)^{-1}\le |N_r|\le k^{-1}\}}dr\Big|H_s,K_s\right]\\
&\qquad\le C \sum_{k=2}^\infty\log^2(k+1)\left(\frac{1}{k^{8/5}}\wedge (t-s)\right)\\
&\qquad\le C(t-s)\sum_{k=2}^{N(s,t)} \log^2(k+1)+C\sum_{k=N(s,t)}^\infty\log^2(k+1)
\frac{1}{k^{8/5}},
\end{align*}
where $N(s,t)=[(t-s)^{-5/8}]$, and we have used for the second inequality Lemma \ref{timespent} and $k^{-1}-(k+1)^{-1}\le k^{-2}$. For some constant $C$, for all $k\ge2$,
$$\log^2(k+1)\le C k^{1/5},$$
hence the second term in the last right--hand side is bounded by
$$C\sum_{k=N(s,t)}^\infty k^{-7/5}\le C(t-s)^{1/4}.$$
Now the first term is bounded by
$$C(t-s)^{3/8}\log^2((t-s)^{-5/8})\le C(t-s)^{1/4}.$$
The claimed result is proved. \epf

 We are now in a position to establish the desired uniqueness result.
  \begin{theorem}\label{uniqueUV}
  Equation \eqref{UV} has a unique solution which spends zero time on
  the diagonal. Furthermore the whole sequence $(U^\epsilon,V^\epsilon)$
  converges weakly to this solution as $\epsilon\to0$.
  \end{theorem}
\begin{proof}
Consider any solution of \eqref{UV} which spends zero time on the diagonal. The time--change
defined by \eqref{tc} transforms this process into a solution $(H_t,K_t)$ of \eqref{HKeq}.

Lemma \ref{port} combined with Lemma \ref{Por1.1} shows that, locally in the open positive quadrant,
the law of that solution to \eqref{HKeq} is absolutely continuous
with respect to that of the same SDE, but without drift. That last equation has a unique
weak solution, according to Theorem 7.2.1 in \cite{SV}. Now any solution to \eqref{HKeq}
coincides with the one constructed via Girsanov's theorem in Theorem 1.1 of \cite {Por}
whose assumptions clearly hold in our case
(alternatively, Girsanov's theorem could also be used thanks to a simplified version of
Lemma \ref{port} together with Lemma \ref{Por1.1} again).

Now that we have a unique weak solution $(H_t,K_t)$ to \eqref{HKeq}, we can time--change
it back to the original $(H_t,K_t)$, i.e.
defining
$A'_t=\int_0^t\frac{ds}{F(H_s,K_s)}$, $\eta'_t=\inf\{s>0,\ A_s>t\}$, $U_t=H_{\eta'_t}$, $V_t=K_{\eta'_t}$.
Indeed a weaker version (without the power 2 of the logarithm) of Lemma 6.9 shows that
$A'_t<\infty$, for all $t<\infty$. Clearly $A'_t\to\infty$, as $t\to\infty$. So this defines
$(U_t,V_t)$ for all $t\ge0$, and  this process coincides with the arbitrary solution which spends
zero time on the diagonal. But the law of that process is uniquely characterized as being
the time--change of the unique weak solution of \eqref{HKeq}.

Finally, it follows from this conclusion and Theorem \ref{accumul} that all accumulation points
of the collection $\{(U^\epsilon,V^\epsilon),\ \epsilon>0\}$
have the same law, hence the whole sequence converges. This proves the stated result.
\end{proof}

From now on $(U_t,V_t)$ will always refer to the process whose law has just been uniquely characterized.

Since the martingale problem associated to \eqref{HKeq} is well posed, $(H_t,K_t)$ is a Markov process,
and from Theorem 6.3 in \cite{Var} (this theorem is stated in dimension 1, but exactly the same argument works in our case),
so is its time--change $(U_t,V_t)$. We call $Q_t$ the Markov semigroup associated to that process, which is defined
 for $\phi\colon \R^2_+ \rightarrow \R$ by
\begin{align}\label{MarkovQ}
  (Q_t\phi)(u,v)=\E_{(u,v)}\big[\phi(U_t,V_t)\big].
\end{align}

\begin{remark}
  We note that in both cases $\sigma_1=\sigma_2$ and
  $\sigma_1\not=\sigma_2$, the law of the uniquely characterized
  solution of \eqref{UV} and that of a non--degenerate SDE in the
  quadrant $\R_+\times\R_+$ are equivalent.

This is in sharp contrast with the result one would get if the
diffusion coefficient would degenerate in a more regular way on the
diagonal (e.g. it would be Lipschitz). In the latter case, in the case
$\sigma_1=\sigma_2$, the solution would stay on the diagonal once it
has hit it. In the case $\sigma_1\not=\sigma_2$, the solution would
stay in the set $u\ge v$ or $v\ge u$, depending upon the sign of
$\sigma^2_1-\sigma^2_2$, after having hit once the diagonal.

Often the period around an orbit diverges like $1/\log(\rho)$ while
approaching a heteroclinic cycle or a homoclinic orbit. Here $\rho$ is
the distance from the limiting orbit. This often leads to
coefficients which vanish  very slowly. This is the situation in our
setting.  Hence while it may seem esoteric at first,   in fact it is
likely to be generic in many settings.

\end{remark}

\subsection{Longtime behavior of $(U,V)$}

Unlike the pair $(U_t^\epsilon,V_t^\epsilon)$, the pair $(U_t,V_t)$
constructed from $(H_t,K_t)$ in the previous section form a Markov
process. Hence, we can speak of an invariant probability measure
for the Markov semigroup $Q_t$.

Observe that
\begin{align*}
  d( U_t +V_t) = \big[ a  - 2(U_t+V_t)\big] dt + dM_t
\end{align*}
where $a>0$. $M_t$ is a continuous local Martingale satisfying
\begin{align*}
  d\langle M \rangle_t \leq c (U_t + V_t) dt
\end{align*}
for some positive $c$ (recall that $\Gamma(u,v)\le u\wedge v$). Hence the following result follows from
Lemma~\ref{l:boundingLemma}.
\begin{proposition}\label{l:boundedUV} For any $p \geq 1$, there exists a constant $C(p)$
  so that
  \begin{align*}
  \sup_{t \geq0} \E \big[ U^p_t + V_t^p] \leq C(p) [1 +    U^p_0+ V_0^p]\,.
  \end{align*}
  \end{proposition}

We shall also need the an analogous result for the $(H,K)$ process. We
could prove stronger results
but the following will be sufficient for our purposes.
\begin{lemma}\label{estimH}
  \begin{equation*}
    \sup_{t\ge0}\E\left(|H_t|^2+|K_t|^2\right)<\infty\,.
  \end{equation*}
\end{lemma}
\begin{proof}
It suffices to treat $H_t$. Defining $G_t=[(H_t-\sigma_1^2)_+]^2$, we note that
$$\frac{d}{dt}\E(G^2_t)\le a-b\E[G_t^2],$$
for some positive constants $a$ and $b$. The result follows readily.
\end{proof}

\begin{lemma}\label{integrable} The process $(H,K)$ possesses a unique
  invariant probability measure $\nu$. Additionally, $\nu$ has a
  denisty which is
  everywhere positive in the interior of $(0,\infty) \times(0,\times)$
  and
  \begin{align*}
    \frac{1}F\in L^1(\nu).
  \end{align*}
where $F$ was defined in \eqref{eqF} and was  used in the time change
between $(U,V)$ and $(H,K)$.
\end{lemma}
\begin{proof}[Proof of Lemma~\ref{integrable}]
Lemma~\ref{estimH} gives the needed tightness of the averaged
transition densities of $(H,K)$ to employ the Krylov--Bogolyubov theorem to
show the existence of an invariant measure (just as we did in the
proof of Corollary~\ref{cor:xiInvMeasure}).
 Since $(H,K)$ was  obtained by a Girsanov transformation from a non degenerate diffusion with zero drift,
its transition probabilities have a density which is positive in the
open positive quadrant. This immediately implies that any invariant
measure must have a density which respect to Lebesgue which is
positive in the
open positive quadrant. This in turn implies that there can only be
one invariant probability measure which we will henceforth denote by $\nu$.

The Birkoff ergodic theorem and the fact that $\nu$ and any transition
probability have densities which are positive in the  open positive quadrant
imply that for any initial distribution and any $\phi \in L^1(\nu)$
\begin{align}\label{simpleLLN}
  \frac1t \int_0^t \phi(H_s,K_s) ds  \rightarrow \int \phi d\nu
  \quad\text{a.s}\quad\text{as} \quad t \rightarrow \infty.
\end{align}

It is not hard to see that whether $\int (1/F)d\nu<\infty$ or $=\infty$, in both cases, as $t\to\infty$,
\begin{equation*}
  \frac{1}{t}\int_0^t\frac{ds}{F(H_s,K_s)}\to \int \frac1F \,d\nu\quad\text{a.s.}
\end{equation*}
The case  $\int (1/F)d\nu=\infty$ follows as follows. For any $M>0$,
$$\int\left(\frac{1}{F}\wedge M\right)d\nu\le\liminf_{t\to\infty}\frac{1}{t}\int_0^t\frac{ds}{F(H_s,K_s)},$$
which
implies by monotone convergence that
$$+\infty\le\liminf_{t\to\infty}\frac{1}{t}\int_0^t\frac{ds}{F(H_s,K_s)},$$
hence the result.

Consequently, in order to prove that $1/F\in L^1(\nu)$,  as a consequence of Fatou's Lemma,
all we have to show is that  there exists $C>0$ and $T>0$ such that for all $t\ge T$,
\begin{equation}\label{bound1/F}
\E\int_0^t \frac{ds}{F(H_s,K_s)}\le Ct.
\end{equation}

It follows from Lemma \ref{estimH} and Lemma \ref{tric} that all \eqref{bound1/F} will follow
from the fact that for some constant $C$ and all $t\ge1$,
\begin{align*}
  \int_0^t\E\ \left[\big|\log|N_s|\big|{\bf1}_{\{|N_s|\le1\}}\right]ds\le Ct.
\end{align*}
Consider the function $\varphi\in C^1(\R)\cap W^{2,\infty}(\R)$ defined as
\begin{align*}
  \varphi(x)=\begin{cases}
\frac{7}{4}-x,&\text{if $x<-1$;}\\
\frac{x^2}{4}(3-2\log(|x|)),&\text{if $-1\le x\le1$;}\\
x-\frac{1}{4},&\text{if $x>1$.} \end{cases}
\end{align*}
We then have
$$\varphi'(x)=\begin{cases}
-1,&\text{if $x<-1$;}\\
x(1-\log(|x|)),&\text{if $-1\le x\le1$;}\\
1,&\text{if $x>1$;} \end{cases}$$
and
$$\varphi''(x)=(-\log(|x|)_+.$$
We now deduce from It\^o's formula
\begin{multline*}
\E\int_0^t\big|\log|N_s|\big|{\bf1}_{\{|N_s|\le1\}}\overline{g}(H_s,K_s)ds=\\
2\E[\varphi(N_t)-\varphi(N_0)]
-2\E\int_0^t\frac{\varphi'(N_s)}{F(H_s,K_s)}[f_1(H_s)-f_2(K_s)]ds.
\end{multline*}
The result now follows easily from Lemma \ref{estimH} and the fact that the process
$F^{-1}(H_s,K_s)\varphi'(N_s)$ is bounded.
\end{proof}

  \begin{theorem}\label{thm:invMeasureUV}
    The semigroup $Q_t$ defined in the previous section possesses a
    unique invariant probability measure $\lambda$. Furthermore,
    $\lambda(du, dv)=\rho(u, v)\,du\,dv$ where $\rho$ is continuous
    away from the diagonal and positive in the positive open
    quadrant. Lastly $\rho(u,v) \rightarrow \infty$, as $ |u-v|
    \rightarrow 0$ in the positive quadrant.
  \end{theorem}
  \begin{proof}[Proof of Theorem~\ref{thm:invMeasureUV}]
Let $\rho(u,v)$ be the density of the unique invariant measure $\nu$
as guaranteed by Lemma~\ref{integrable} . The same lemma also states that $\rho$
is positive in the open
positive quadrant.

Now,
 using the notations from the proof of Theorem~\ref{uniqueUV}, for any
 measurable and locally bounded
 $f:(0,\infty)\times(0,\infty) \rightarrow \R$ one has
\begin{align*}
  \int_0^tf(U_s,V_s)ds&=\int_0^t f(H_{\eta'_s},K_{\eta'_s})ds\\
  &=\int_0^{\eta'_t}\frac{f(H_r,K_r)}{F(H_r,K_r)}dr\\
  \frac{1}{t}\int_0^t f(U_s,V_s)ds&=\frac{\eta'_t}{t}\times\frac{1}{\eta'_t}\int_0^{\eta'_t}\frac{f(H_r,K_r)}{F(H_r,K_r)}dr
  \end{align*}
If one assumes that $f$ is bounded then Lemma~\ref{integrable} ensures
that $f/F \in L^1(\nu)$ and that for any initial distribution
  \begin{align*}
    \lim_{t\to\infty}\frac{1}{t}\int_0^t
  f(U_s,V_s)ds&=\Big(\lim_{t\to\infty}\frac{\eta'_t}{t}\Big)\times
  \int \frac{f}{F}\, d\nu
\quad \text{a.s}
  \end{align*}
where the convergence of
$\frac{1}{\eta'_t}\int_0^{\eta'_t}\frac{f(H_r,K_r)}{F(H_r,K_r)}dr$ to
$\int \frac{f}{F}\, d\nu$ is ensured by \eqref{simpleLLN}.
The same computation with $f\equiv1$  permits one to conclude that
\begin{equation*}
  \lim_{t\to\infty}\frac{\eta'_t}{t}=\Big(\int\frac{1}{F}\,d\nu\Big)^{-1}\,.
\end{equation*}
In summary, we have that for any $f$ measurable and bounded, for any initial probability measure,
\begin{align}
  \label{eq:limitlambda}
  \lim_{t\to\infty}\frac{1}{t}\int_0^tf(U_s,V_s)ds=\frac{\int
  \frac{f}{F} \,d\nu}{\int \frac{1}{F}\, d\nu}\,.
\end{align}
This is enough to show that our process $(U_t,V_t)$ has the unique invariant probability measure
\begin{equation*}
  \lambda(du,dv)=\frac{F^{-1}(u,v)\rho(u,v)dudv}{\int_{\R^2_+}F^{-1}(u,v)\rho(u,v)dudv}\,.
\end{equation*}
Indeed, to see uniqueness, let $\lambda'$ be any ergodic invariant probability
measure for $(U_t,V_t)$. Then for any bounded $f$, the Birkoff ergodic
theorem implies that
\begin{align*}
  \lim_{t \rightarrow \infty}\frac1t \int_0^t f(U_s,V_s)ds = \int f d\lambda' \quad\text{a.s}
\end{align*}
for $\lambda'$-almost every initial $(U_0,V_0)$. Combining this with
\eqref{eq:limitlambda} implies that $\int f d\lambda' = \int f
d\lambda$ for all bounded $f$ which in turn implies
$\lambda=\lambda'$. Since any invariant measure can be decomposed into
ergodic invariant measures the  uniqueness of $\lambda$ is proved.
\end{proof}

\section{The Deterministic Dynamics}\label{sec7}
\label{sec:detDynam}

We now investigate more fully the deterministic dynamics given in
\eqref{epsZero} and  obtained
by  formally setting $\epsilon=0$ in  \eqref{eps}.
As already mentioned, \eqref{epsZero} has two conserved quantities
$(u,v)=\Phi(\xi_0)$ which are constant on any given orbit. If
$\xi_0=(X_0,Y_0,Z_0)$ then $u=2x^2+z^2$ and $v=2y^2+z^2$ give two
independent equations. Since we are working in three dimensions, the
locus of the solutions, which contains the points in the orbits, is a one--dimensional curve.
We undertake this study since the $\frac1\epsilon \BB$ term in \eqref{fast} implies that on the fast
timescale the solution  will make increasingly many turns very near a
deterministic orbit of \eqref{epsZero}, before the stochastic or
dissipative  terms
cause appreciable diffusion or drift from the current deterministic orbit.

\subsection{Structure of orbits}
If $u \neq v$ then the orbit is a simple periodic orbit which is
topologically equivalent to a circle.  In this case, there are two disjoint orbits
which are solutions.
If $u>v$, one such orbit is given by
\begin{align*}
\Gamma_{u,v}^+=
\Big\{ (\sqrt{ \tfrac{u-z^2}2},  \pm\sqrt{\tfrac{v-z^2}2}, z  )  \colon  z\in  [-\sqrt{v},\sqrt{v}]  \Big\},
\end{align*}
and another by
\begin{align*}
\Gamma_{u,v}^-=
\Big\{ (-\sqrt{ \tfrac{u-z^2}2},  \pm\sqrt{\tfrac{v-z^2}2}, z  )  \colon  z\in  [-\sqrt{v},\sqrt{v}]  \Big\} .
\end{align*}
Similarly if $v>u$ then the corresponding orbits are given by
\begin{align*}
\Gamma_{u,v}^+=&
\Big\{ ( \pm\sqrt{\tfrac{u-z^2}2},  \sqrt{\tfrac{v-z^2}2}, z  )  \colon  z\in  [-\sqrt{u},\sqrt{u}]  \Big\}, \\
\Gamma_{u,v}^-=&
\Big\{ (\pm\sqrt{ \tfrac{u-z^2}2}, - \sqrt{\tfrac{v-z^2}2}, z  )  \colon  z\in  [-\sqrt{u},\sqrt{u}]  \Big\}.
\end{align*}

Whether $u>v$ or $v>u$
is enough information to localize a given orbit to one of two orbits on
sphere of radius $\sqrt{(u+v)/2}$. The remaining piece of information is contained in the
sign of the function defined by
\begin{align}\label{defSigma}
\sn(x,y,z)=\sign\left(\one_{|x|> |y|}x+\one_{|y|>|x|}y\right)
\end{align}
The value of $\sn$ corresponds to the sign decorating the
$\Gamma_{u,v}^\pm$. Hence if one starts from the initial condition
$(x,y,z)$ such that the $(u,v)$ computed from these orbits satisfies
$u\neq v$ then the deterministic dynamics will trace the set $\Gamma_{u,v}^\sn$.

The exception to being topologically equivalent to a
circle are the lines of fixed points given by $\{(0,0,z)\colon z \in
\R\} $, $\{(x,0,0)\colon x \in \R\} $, and $\{(0,y,0)\colon y \in \R\} $
and the heteroclinic orbits which connect them which are contained
in the locus of points where $u=v$. For a given such choice there are four
heteroclinic orbits given by
\begin{align*}
\mathcal{H}_u^{(1)}&=\Big\{ (\sqrt{\tfrac{u- z^2}2},\sqrt{\tfrac{u- z^2}2},z) \colon z\in
[-\sqrt{u},\sqrt{u}]\Big\} \\
\mathcal{H}_u^{(2)}&=\Big\{ (\sqrt{\tfrac{u- z^2}2},-\sqrt{\tfrac{u- z^2}2},z)
\colon z\in [-\sqrt{u},\sqrt{u}]\Big\}\\
\mathcal{H}_u^{(3)}&=\Big\{ (-\sqrt{\tfrac{u- z^2}2},-\sqrt{\tfrac{u- z^2}2},z) \colon z\in
[-\sqrt{u},\sqrt{u}]\Big\}
\\
\mathcal{H}_u^{(4)}&=\Big\{ (-\sqrt{\tfrac{u- z^2}2},\sqrt{\tfrac{u- z^2}2},z)
\colon z\in [-\sqrt{u},\sqrt{u}]\Big\}
\end{align*}

These heteroclinic orbits split each sphere into four regions
which contain closed orbits of finite period. The following set limits
hold
\begin{align*}
  \lim_{v \rightarrow u^-} \Gamma_{u,v}^+&=\lim_{u\rightarrow v^+}
  \Gamma_{u,v}^+ = \mathcal{H}_u^{(1)} \cup \mathcal{H}_u^{(2)}\\
 \lim_{v \rightarrow u^+} \Gamma_{u,v}^+&=\lim_{u\rightarrow v^-}
  \Gamma_{u,v}^+ = \mathcal{H}_u^{(1)} \cup \mathcal{H}_u^{(4)}\\
 \lim_{v \rightarrow u^-} \Gamma_{u,v}^-&=\lim_{u\rightarrow v^+}
  \Gamma_{u,v}^- = \mathcal{H}_u^{(3)} \cup \mathcal{H}_u^{(4)}\\
 \lim_{v \rightarrow u^+} \Gamma_{u,v}^-&=\lim_{u\rightarrow v^-}
  \Gamma_{u,v}^- = \mathcal{H}_u^{(2)} \cup \mathcal{H}_u^{(3)}
\end{align*}
In contrast to the case when $u\neq v$, the orbits starting from a
given point $(x,y,z)$ do not converge to one of these unions of
heteroclinic trajectories since any
given orbit is restricted to a single heteroclinic trajectory. This
could be a point of concern, but we will see in the next section that it
does not pose a problem, which is an interesting and important feature
of this model.

\subsection{Symmetries and their implications}
\label{sec:symmetry}
Defining $\sym_e\colon \R^3 \rightarrow \R^3$ by
$\sym_e(x,y,z)=(y,x,z)$ and $\sym_\pm\colon \R^3 \rightarrow \R^3$ by $\sym_\pm(x,y,z)=(-x,-y,z)$,
observe that if $\xi_t$ is a
solution to \eqref{epsZero}
then so are $\sym_e(\xi_t)$ and
$\sym_\pm(\xi)$. This
implies that
$\Gamma_{u,v}^-=\sym_\pm (\Gamma_{u,v}^+)$ and $\Gamma_{u,v}^+=\sym_e(\Gamma_{v,u}^+)$, and that
  if $\mu$ is an invariant probability measure for $P_t$
then necessarily $\mu \sym_e^{-1}$ and
$\mu \sym_\pm^{-1}$ are also invariant probability measures for $P_t$.

The situation for the stochastic dynamics given in \eqref{eps} is the
same for $\sym_\pm$  but depends on the choice of $\sigma_1$ and
$\sigma_2$ for $\sym_e$. In all cases $\sym_\pm(\xi_t^\epsilon)$ is a
solution (for a different Brownian motion) if $\xi_t^\epsilon$ is a
solution. However  $\sym_e(\xi_t^\epsilon)$ is  again  a solution if
$\xi_t^\epsilon$ is one only when $\sigma_1=\sigma_2$. In any case, we
have the following observation which we formulate as a proposition for
future reference.
\begin{proposition}\label{symInvMeasure}  Let $\sym
  \colon \R^3 \rightarrow \R^3$ be a map such that $\sym(\xi_t^\epsilon)$ is a solution (for possibly a different
  Brownian motion) whenever $\xi_t^\epsilon$ is a solution, then
  $\mu^\epsilon=\mu^\epsilon\sym^{-1}$ where $\mu^\epsilon$ is the
  unique invariant probability measure of $P_t^\epsilon$ guaranteed by Theorem~\ref{invMeasureEpsPos}.
\end{proposition}
\begin{proof}[Proof of Proposition~\ref{symInvMeasure}]
  As before it is clear that $\mu^\epsilon\sym^{-1}$ is again an
  invariant probability measure, however we know that $\mu^\epsilon$ is the
  unique invariant probability measure given the assumptions on the
  $\sigma$'s. Hence we conclude that $\mu^\epsilon\sym^{-1}=\mu$.
\end{proof}

\subsection{Averaging along the deterministic trajectories}\label{avgs}
Since the separation of time scale between the fast and slow dynamics
leads to the averaging of the coefficients of
$(U^\epsilon,V^\epsilon)$ equation around the deterministic orbits we
now discuss averaging along the deterministic orbits in general. After
this we will define the function $\Lambda$
whose asymptotics was described in Proposition \ref{prop:Lambda}.

Given a function $\psi:\R^3 \rightarrow \R$, we define
\begin{align}\label{avg}
 (\mathcal{A}\psi)(\xi)=\lim_{t\rightarrow \infty} \frac1t\int_0^t
(\psi\circ \varphi_s)(\xi) ds.
\end{align}
 Notice
that $\mathcal{A}\psi$ is again a function from $\R^3 \rightarrow
\R$ and that it is constant on the connected components of the level sets of $(u,v)$.

\subsubsection{Averaging when $u\neq v$}\label{sec:avg}
Let $(u,v)=\Phi(\xi)$.
If $u\neq v$ then $\xi$ lies on a
periodic orbit of  finite period. Letting $\tau$ denote the period,  one has
\begin{align*}
 (\mathcal{A}\psi)(\xi)=\frac1\tau \int_0^\tau
(\psi\circ \varphi_s)(\xi) ds
\end{align*}

To obtain a more explicit representation for the averaging operation
we will switch to an angular variable $\theta$. Given any positive $u$
and $v$,  for $\theta \in [0,2\pi]$ we parametrize $z$ by
$z(\theta)=\sqrt{u\wedge v}\sin(\theta)$. To define the other
coordinates we introduce the following auxiliary angles
\begin{align*}
\phi_1(\theta)&=
\begin{cases}
 \arcsin\big(\sqrt{\frac{v}{u}} \sin \theta\big) & u > v\\
 \theta & u \leq v
\end{cases}, &
\phi_2(\theta)&=  \begin{cases}
 \theta &  u \geq v \\
\arcsin\big(\sqrt{\frac{u}{v}} \sin \theta\big) & u < v
\end{cases}  \,.
\end{align*}
and set
$x(\theta)=\sqrt{\frac{u}2}\cos(\phi_1(\theta))$, and
$y(\theta)=\sqrt{\frac{v}2}\cos(\phi_2(\theta))$.
Putting everything together  we have that the
trace of the trajectory starting at
$(\sqrt{\tfrac{u}2},\sqrt{\tfrac{v}2},0)$ is given by
\begin{align*}
  \Gamma_{u,v}^+= \big\{\orbit_{u,v}(\theta)\colon \theta
  \in[0,2\pi] \big\}
\end{align*}
where $\orbit_{u,v}(\theta)\eqdef\big( x(\theta),y(\theta),z(\theta)\big) $.
As already discussed depending on weather $u>v$ or $v>u$ this
represents a closed orbit on the sphere of radius $\sqrt{(u+v)/2}$
which rotates around respectively either the $x$-axis in the positive
$x$ half space or the $y$-axis in the positive $y$ half space. The
orbits in the negative half space are given by $\Gamma_{u,v}^-=\sym_\pm(\Gamma_{u,v}^+)$.

To define the occupation measure on these orbits we define a third
auxiliary angle
\begin{align*}
\phi_{u,v}(\theta)&=  \arcsin\big(\sqrt{\tfrac{u \wedge v}{u \vee v}}
\sin \theta \big) =
\begin{cases}
 \phi_1(\theta) & u> v\\
\theta & u=v\\
 \phi_2(\theta) & u<v
\end{cases}
\end{align*}
For $u \neq v$, we define a probability measure on $\R^3$ by
\begin{align}\label{ocupation}
\nu_{u,v}^{+}(dx\,dy\,dz)=
\int_0^{2\pi}   \frac{K_{u,v}}{|\cos( \phi_{u,v}(\theta))|}
\delta_{\orbit_{u,v}(\theta)}( dx\,dy\,dz)d\theta
\end{align}
where $K_{u,v}^{-1}=\int_0^{2\pi}\frac1{|\cos(
  \phi_{u,v}(\theta))|}d\theta$. We let
$\nu_{u,v}^{-}(dx\,dy\,dz)= \nu_{u,v}^{+}\sym_\pm^{-1}$. For $u=v$, we define
$\nu_{u,u}^{\pm}(dx\,dy\,dz)=\delta_{(0,0,\pm\sqrt{u})}(dx\, dy\, dz)$. Each of these probability measures is supported
on the corresponding set $\Gamma_{u,v}^+$ or  $\Gamma_{u,v}^-$. It is
straightforward to see that for any $\psi:\R^3 \rightarrow \R$ and  $(x,y,z) \in \R^3$ such that
$|x|\neq |y|$ one has
\begin{align}\label{avgNu}
 ( \mathcal{A}\psi)(\xi)=\int_{\R^3} \psi(\eta) \nu_{u,v}^\s(d\eta)
\end{align}
where  $\xi=(x,y,z)$, $(u,v)=\Phi(\xi)$ and $\s=\sn(\xi)$ where
$\sn$ was defined in \eqref{defSigma}.

\subsubsection{Definitions of $\Gamma$ and  $\Lambda$}
\label{sec:Lambda}

The central quantities which need to be averaged in the
$(U^\epsilon_t,V^\epsilon_t)$ dynamics, given in equation
\eqref{rho-chi}, are the infinitesimal quadratic variations. They are
given respectively by $16\sigma_1^2 x^2$ and $16\sigma_1^2 y^2$.
From \eqref{eq:Phi}, we have that $x^2=\frac12(u-z^2)$ and
$y^2=\frac12(v-z^2)$. Since $u$ and $v$ are constant along
the deterministic trajectories, this in turn implies that
\begin{align*}
   \mathcal{A}(x^2) = \frac{u-  \mathcal{A}(z^2)}{2}\quad\text{and}\quad    \mathcal{A}(y^2) = \frac{v-  \mathcal{A}(z^2)}{2}\,.
\end{align*}
Since $z^2$ does not depend on the chose of sign in the definition
$\nu_{u,v}^{\pm}(dx\,dy\,dz)$ by defining the single function
\begin{align}\label{eq:Gamma}
  \Gamma(u,v)=  \mathcal{A}(z^2)= K_{u,v} \int_0^{2\pi}
  \frac{(u\vee v)\sin^2( \theta)  }{|\cos( \phi_{u,v}(\theta))|} d\theta
\end{align}
we have access to  all of the averaged quantities we will require.

Clearly the function $\Gamma(u,v)$ is symmetric in $(u,v)$ and can be
written as a function of $u\vee v$ and $u\wedge v$ only. In fact one
see that if one defines
\begin{align*}
  \Lambda(r)&=  K_r\int_0^{2\pi}
  \frac{\sin^2( \theta)  }{|\cos(  \arcsin( r \sin(\theta)))|}
  d\theta,\ \text{ where}\\
   K_r^{-1} &=\int_0^{2\pi}\frac1{|\cos(
  \arcsin(r \sin(\theta)))| }d\theta
\end{align*}
for $r \in [0,1]$ then \eqref{LambdaToGamma} holds. The properties of
$\Lambda$ given in Proposition~\ref{prop:Lambda} follow directly from
this definition, Proposition~\ref{avgNearDiag} in the next subsection,
and its proof.

\subsubsection{Averaging near the diagonal}
\label{sec:aver-near-diag}

\begin{proposition}\label{avgNearDiag} Let $\psi:\R^3 \rightarrow \R$ be a continuous
  function. If $\delta=1-(u \wedge v)/(v\vee u)$ then as
  $|u-v| \rightarrow 0$ (and hence $\delta \rightarrow 0$) while $(u,v)$ remains in a compact set, one has
  \begin{align}
    \label{nuAsy}
    (\nu\psi)(u,v,\sigma) =   \hf( \psi(0,0,\sqrt{u\vee v}) +
    \psi(0,0,-\sqrt{u\vee v}) ) + o(1)
  \end{align}
If in addition for all $u$
\begin{align*}
  C_{u}(\psi) = \int_0^{2\pi}
  \frac{\psi(\sqrt{u}\cos(\theta),\sqrt{u}\cos(\theta),\sqrt{u}\sin(\theta))}{|\cos(\theta)|}
  d\theta < \infty,
\end{align*} then as $u \rightarrow v$ one has
\begin{align*}
     (\nu\psi)(u,v,\sigma) =
     \frac{C_{u}(\psi) }{2|\ln(1-r  )|} + o\big(|\ln(1-r  )|^{-1}\big),
\end{align*}
where $r=\frac{u\wedge v}{v\vee u}$.
\end{proposition}

\begin{remark}
  The asymptotic expansion given in Proposition~\ref{prop:Lambda}
  follows from the fact that $C_u(1-z^2)=4u$. The continuity
  properties follow from the formulas and the fact that the values at
  the ends of the intervals are finite.
\end{remark}

\begin{proof} \cc{This proof needs to be fixed}

We will begin by exploring the asymptotics of the
  constant $K_r$ which equals $K_{u,v}$ when $r=\frac{u\wedge v}{v\vee u}$. Making the change of variables $\alpha=\sin
  \theta$ followed by
  $\beta^2=\sqrt{r}\alpha^2$, one has
  \begin{align*}
    K_r^{-1}&=\int_0^1 \frac{4}{\sqrt{(1-r \alpha^2)(1-\alpha^2)}}
    d\alpha=\frac{4}{r^\frac14} \int_0^{r^\frac14}  \frac{1}{\sqrt{(1-r^{\frac12} \beta^2)(1-r^{-\frac12}\beta^2)}}
    d\beta   \\
    &=\frac{4}{r^\frac14} \int_0^{r^\frac14} \frac{1}{1-\beta^2}\frac{1}{\sqrt{1-\gamma(r)\frac{\beta^2}{(1-\beta^2)^2}}}
    d\beta  ,
  \end{align*}
where $\gamma(r)=r^{\frac12}+r^{-\frac12}-2$.
Now since for all $\beta,r \in (0,1]$ with $\beta \leq r$
\begin{align*}
0\leq   \gamma(r)\frac{\beta^2}{(1-\beta^2)^2} \leq \gamma(r)\frac{r^2}{(1-r^2)^2} \leq\frac1{16}
\end{align*}
 we have
\begin{align}\label{Kbounds}
\frac{4}{r^{\frac14}} \mathrm{arctanh}(r^\frac14)  \leq K_r^{-1} \leq \frac{16}{\sqrt{15}}\frac1{r^{\frac14}} \mathrm{arctanh}(r^\frac14) ,
\end{align}
again for all $r \in (0,1]$. Furthermore it is clear that
\begin{align}\label{KAsy}
  \lim_{r\rightarrow 1} K_r|\ln(1-r)| =\frac12\,.
\end{align}

Now
\begin{align}\label{int1}
   (\nu
   \psi)(u,v,\sigma)=K_{u,v}\int_0^{2\pi}\frac{\psi(x(\theta),y(\theta),z(\theta))
   }{|\cos( \phi(\theta))|} d\theta \,.
\end{align}
As $|u-v|\rightarrow 0$ and hence $r \rightarrow 1$, this integral concentrates around the two points
$\theta$ equal $\pi/2$ and  $3\pi/2$ since around these points $|\cos(
\phi(\theta))| \rightarrow 0$ as  $r\rightarrow 1$. At these
points $(x(\theta),y(\theta),z(\theta))$ converges to
$(0,0,\sqrt{u\vee v})$ and $(0,0,-\sqrt{u\vee v})$
respectively. Around these points we have one behavior and away from
the another. Consider the following representative portion of the
integral which will converge to  $\hf\psi(0,0,\sqrt{u\vee v})$. Fixing
any sufficiency small $\epsilon>0$, we define $a=a(\epsilon)$ so that $\sin(\pi/2
-a) =\sin(\pi/2
+a)  =1-\epsilon$. Then
\begin{align*}
  K_{u,v}\int_0^{\pi}&\frac{\psi(x(\theta),y(\theta),z(\theta))
   }{|\cos( \phi_{u,v}(\theta))|} d\theta =  K_{u,v}\int_0^{\frac{\pi}2-a}\frac{\psi(x(\theta),y(\theta),z(\theta))
   }{|\cos( \phi_{u,v}(\theta))|} d\theta \\&+  K_{u,v}\int_{\frac{\pi}2-a}^{\frac{\pi}2+a}\frac{\psi(x(\theta),y(\theta),z(\theta))
   }{|\cos( \phi_{u,v}(\theta))|} d\theta +  K_{u,v}\int_{\frac{\pi}2+a}^{\pi}\frac{\psi(x(\theta),y(\theta),z(\theta))
   }{|\cos( \phi_{u,v}(\theta))|} d\theta\,.
\end{align*}
The remaining half of the integral in \eqref{int1} will converge to
$\hf\psi(0,0,-\sqrt{u\vee v})$ in a completely analogous fashion. The
first and third integral behave the same. We consider the first. If,
as before, we have $r=\frac{u\wedge v}{v\vee u}$  and then make the change of variables $\alpha=\sin
  \theta$ followed by
  $\beta^2=\sqrt{r}\alpha^2$ to obtain
\begin{align*}
   K_{u,v}\Big|\int_0^{\frac{\pi}2-a}\tfrac{\psi(x(\theta),y(\theta),z(\theta))
   }{|\cos( \phi_{u,v}(\theta))|} d\theta\Big| &\leq \tfrac{ K_{u,v}\|\psi\|_\infty}{r^{\frac14}}
   \int_0^{(1-\epsilon)r^\frac14}  \tfrac{1}{1-\beta^2}\tfrac{1}{\sqrt{1-\gamma(r)\frac{\beta^2}{(1-\beta^2)^2}}}
    d\beta \\
&\leq C K_r \|\psi\|_\infty \mathrm{arctanh}((1-\epsilon)r^{\frac14})\,.
\end{align*}
 By the asymptotics on $K_r$ given  in \eqref{Kbounds}, this
 goes to zero since $\epsilon>0$ as $|u-v| \rightarrow 0$ and  hence $r \rightarrow 1$.

Now as $r \rightarrow 1$ one has
 \begin{align*}
    K_r\int_{\frac{\pi}2-a}^{\frac{\pi}2+a}\frac{\psi(x(\theta),y(\theta),z(\theta))
   }{|\cos( \phi(\theta))|} d\theta\approx& \psi(0,0,\sqrt{u\vee v})  \frac{2K_r}{r^\frac14}  \int_{(1-\epsilon)r^{\frac14}}^{r^{\frac14}} \tfrac{1}{1-\beta^2}
    d\beta\\
\approx&\hf\psi(0,0,\sqrt{u\vee v}) \,.
 \end{align*}
The last conclusion follows directly from the assumed finiteness of
$C_u(\psi)$ and the asymptotics of $K_r$ as $r \rightarrow 1$ as  $|u-v| \rightarrow 0$.
in \eqref{KAsy}.
\end{proof}

\subsection{The Ergodic Invariant Measures}
\label{sec:ergod-invar-meas}
The set of ergodic invariant probability measures is the set of extremal
invariant probability measures. The extremal measures are those which can not be
decomposed. Clearly this corresponds to the collection of  the occupancy measures of each periodic orbit along
with the delta measures sitting on each of the fixed points
$(0,0,\sqrt{u})$ and $(0,0,-\sqrt{u})$. These are precisely the
measures $\nu_{u,v}^\pm$ defined in Section~\ref{sec:avg}. Since the union
of these orbits and fixed points covers all the space except for the
heteroclinic connections which cannot support an invariant
probability measure. Hence
we have identified all  the ergodic probability measures.

We summarize this discussion in the following result.
\begin{proposition}\label{prop:ergodicMeasures}
  The set of ergodic invariant probability measure of \eqref{epsZero} consists
  precisely of
  \begin{align*}
    \{ \nu_{u,v}^+, \nu_{u,v}^- : u,v >0\}\,.
  \end{align*}
\end{proposition}
Given $(u,v) \in \R_+^2$, we define the probability measure
$\nu_{u,v}$ on $\R^3$ by
\begin{align}\label{eq:nu_uv}
  \nu_{u,v}(d\xi)= \frac12  \nu_{u,v}^+(d\xi) + \frac12\nu_{u,v}^-(d\xi)
\end{align}
where $\nu_{u,v}^\pm(d\xi)$ we defined in  \eqref{ocupation} and the
text below it.

The following corollary of Proposition~\ref{prop:ergodicMeasures} will
be central to the proof of the convergence of $\mu^\epsilon$ to a
unique limiting measure.
\begin{corollary}\label{cor:uniqRepresentationThm}
  Any invariant  probability measure $m$ for \eqref{epsZero} which satisfies $m
  \sym_\pm^{-1} = m$ can be represented as
  \begin{align*}
    m(dx\,dy\,dz) = \int_{[0,\infty)^2} \nu_{u,v}(dx\,dy\,dz)\, \gamma(du\,dv)
  \end{align*}
  for some probability measure $\gamma$ on $(0,\infty)^2$. Furthermore
  the measure $\gamma$ is unique. Conversely, a probability measure which is
  invariant for   \eqref{epsZero} and satisfies $m
  \sym_\pm^{-1} = m$ is uniquely specified by the measure $\gamma=m
  \Phi^{-1}$.
\end{corollary}
\begin{proof}[Proof of Corollary~\ref{cor:uniqRepresentationThm}]
  The ergodic decomposition theorem \cite{Sinai87} implies that there exists a
  unique pair of measures $(\gamma^+,\gamma^-)$ so that the total mass
  of $\gamma^++\gamma^-$ is one and
  \begin{multline*}
    m(dx\,dy\,dz) = \int_{[0,\infty)^2} \nu_{u,v}^+(dx\,dy\,dz)\,
    \gamma^+(du\,dv)\\+ \int_{[0,\infty)^2} \nu_{u,v}^-(dx\,dy\,dz)\, \gamma^-(du\,dv)\,.
  \end{multline*}
  Now since $m \sym_\pm^{-1} = m$, $\nu_{u,v}^-=\nu_{u,v}^+
  \sym_\pm^{-1}$ and  $\nu_{u,v}^+=\nu_{u,v}^- \sym_\pm^{-1}$, we have that
  \begin{multline*}
    m(dx\,dy\,dz) = \int_{[0,\infty)^2} \nu_{u,v}^-(dx\,dy\,dz)\,
    \gamma^+(du\,dv)\\+ \int_{[0,\infty)^2} \nu_{u,v}^+(dx\,dy\,dz)\, \gamma^-(du\,dv)\,.
  \end{multline*}
Since $\nu_{u,v}^-$ and  $\nu_{a,b}^+$ are mutually singular for all
 choices of positive $u$, $v$, $a$, and $b$, we see that
$\gamma^+=\gamma^-$ and the total mass of both is $\frac12$.  Setting
$\gamma=2 \gamma^+=2\gamma^-$ we see that $\gamma$ is a probability
measure and that
\begin{align*}
    m(dx\,dy\,dz) &= \int_{[0,\infty)^2}
    \big[\tfrac12\nu_{u,v}^+(dx\,dy\,dz)\, +
    \tfrac12\nu_{u,v}^-(dx\,dy\,dz)\big]\, \gamma(du\,dv)\\
    &=\int_{[0,\infty)^2}
   \nu_{u,v}(dx\,dy\,dz)\, \gamma(du\,dv)
\end{align*}
This proves that any invariant $m$ satisfying the symmetry assumption
can be represented as claimed. All that remains is to show is that
$\gamma$ is unique. Let $\widetilde \gamma$ be another probability measure
so that
\begin{align*}
    m(dx\,dy\,dz) =\int_{[0,\infty)^2}
   \nu_{u,v}(dx\,dy\,dz)\, \widetilde \gamma(du\,dv)
\end{align*}
which implies that
\begin{align*}
   m(dx\,dy\,dz) =\int_{[0,\infty)^2}
   \big[\tfrac12\nu_{u,v}^+(dx\,dy\,dz)\, +
    \tfrac12\nu_{u,v}^-(dx\,dy\,dz)\big]\, \widetilde \gamma(du\,dv)
\end{align*}
which  in turn implies that $\frac12 \widetilde \gamma=\gamma^+$ since the
ergodic decomposition is unique. However, this implies $\widetilde
\gamma=\gamma$ as  desired.
\end{proof}

\subsection{The Limiting Fast Semigroup}
We begin with a small detour to think about the limiting dynamics. Its
action on a test function can be understood to instantly assign to
each point on an orbit the
average of the function around the orbit and to each point on the
heteroclinic connection the value of the function at the limiting fixed
point on the $z$-axis.

Recall the definition of $\nu_{u,v}$ from \eqref{eq:nu_uv}, for $\phi\colon \R^3 \rightarrow \R$ we define
$\nu\phi$ by
\begin{align*}
  (\nu \phi)(u,v)& = \int \phi(\xi)\nu_{u,v}(d\xi)\,.
\end{align*}
Recalling the definition of $\Phi$ which maps $\xi$ to $(u,v)$ from
\eqref{eq:Phi}, we note that for any $\rho:\R^2_+ \rightarrow \R$,
\begin{align}
  \label{eq:nuConst}
  \nu(\rho\circ\Phi)(u,v) = \rho(u,v) \,.
\end{align}
Lastly recalling the definition of
$\widetilde P_t^\varepsilon$ from \eqref{eq:Pt}, $Q_t$ from
\eqref{MarkovQ} and let $\lambda$ be the unique invariant probability measure of
$Q_t$ guarantied by Theorem~\ref{thm:invMeasureUV}.
For $\phi\colon \R^3 \rightarrow \R$
we define
\begin{align}\label{P0fastDynamics}
  (\widetilde P_t \phi)(\xi)&=(Q_t\nu  \phi)\circ \Phi(\xi)
\end{align}
\begin{remark}
  If $\phi$ is a test function such that $\phi\circ \sym_\pm=\phi$ or
  $m$ is an initial measure on $\R^3$ such that $m=m  \sym_\pm^{-1}$
  then is not hard to convince oneself that $m\tilde P_t^\epsilon\phi
  \rightarrow m\tilde P_t\phi$ as $\epsilon\rightarrow0$. If one
  neither starts with initial data which has this symmetry nor uses a
  symmetric test function, then things are
  more complicated. The orbit may average with respect to only one of
  the two measure: $\nu_{u,v}^+$ or $\nu_{u,v}^-$. For definiteness
  assume that we are
  on the $\nu_{u,v}^+$ orbit.  We believe that
  when the $(U,V)$-dynamics hits the line $U=V$ then it is essentially
  spending all of its time at $(0,0,\sqrt{u})$ and $(0,0,-\sqrt{u})$.
  (See Proposition~\ref{avgNearDiag}.) With probability $\frac12$ it
  returns to a $\nu_{u,v}^+$ orbit and with probability $\frac12$ it
  enters on to a $\nu_{u,v}^-$ orbit. Hence to describe the
  $\widetilde P_t$ semigroup in the non-symmetric setting, it seems we need to
  add a sequence of independent Bernoulli random variables to make
 decision of whether one should average with respect to the  $+$ or the
 $-$ orbit. Since we are primarily interested in the structure of the
 invariant probability measure we have not tried to make this picture rigorous.
\end{remark}

Let $\lambda$ be the unique invariant probability measure of $Q_t$ and define $\mu=\lambda
\nu$. Observe  that $\mu$ is invariant under $\widetilde P_t$ because
for any bounded $\phi\colon \R^3 \rightarrow \R$ one has
\begin{align*}
 \mu \widetilde P_t \phi &= \lambda\nu (Q_t \nu \phi \circ
 \Phi) = \lambda Q_t \nu \phi
= \lambda  (\nu \phi)=  \mu\phi\,.
\end{align*}
Here the first equality is by definition, the second follows from
\eqref{eq:nuConst}, the third from the invariance of $\lambda$ under
$Q_t$ and the last from the definition of $\mu$.

\section{Convergence of $(U^\eps,V^\eps)$ towards $(U,V)$}\label{sec8}

We now prove the results which were taken for granted in Section~\ref{sec6}, namely that
$\{(U_t^\varepsilon,V_t^\varepsilon)\}_{\eps>0}$ is tight, and that any accumulation point
$(U_t,V_t)$  solves the SDE \eqref{UV}.

\cc{some introductory words}

\subsection{Tightness}
Let us rewrite \eqref{rho-chi} in the form
\begin{equation}\label{uveps}
\left\{
\begin{aligned}
dU^\epsilon_t&=(C_u-2U^\epsilon_t)dt+dM^\epsilon_t\\
dV^\epsilon_t&=(C_v-2V^\epsilon_t)dt+dN^\epsilon_t,
\end{aligned}
\right.
\end{equation}
where $\{M^\epsilon_t,\, t\ge0\}$ and $\{N^\epsilon_t,\, t\ge0\}$ are continuous local martingales
such that
\begin{equation}\label{boundQV}
\frac{d}{dt}\langle M^\epsilon\rangle_t\le CU^\epsilon_t,\quad \frac{d}{dt}\langle N^\epsilon\rangle_t\le CV^\epsilon_t,
\end{equation}
where $C_u$, $C_v$ and $C$ are three positive constants.

We want to show
\begin{proposition}\label{tightness} Suppose that
  \begin{align*}
    \sup_{\epsilon >0}\E\left[(U^\epsilon_0)^2+(V^\epsilon_0)^2\right]<\infty\,.
  \end{align*}
Then the collection of processes $\{(U^\epsilon_t,V^\epsilon_t),\, t\ge0\}_{\epsilon>0}$ is
tight in $C([0,+\infty);\R^2)$.
\end{proposition}

In light of \eqref{uveps} and \eqref{boundQV}, the following needed
first step follows from Lemma~\ref{l:boundingLemma}.
\begin{lemma}\label{moments}
Under the condition of Proposition \ref{tightness},
$$\sup_{\epsilon>0}\sup_{t\ge0}\E\left[(U^\epsilon_t)^2+(V^\epsilon_t)^2\right]<\infty.$$
\end{lemma}
We can now proceed with the proof of tightness.
\begin{proof}[Proof of Proposition~\ref{tightness}]
We prove tightness of $U^\epsilon$ only, $V^\epsilon$ being treated completely similarly.
We have
$$U^\epsilon_t=\frac{C_u}{2}+\left(U^\epsilon_0-\frac{C_u}{2}\right)e^{-2t}+e^{-2t}\int_0^te^{2s}dM^\epsilon_s.$$
Clearly the first two terms on the right are tight in $C([0,\infty))$, since the collection of $\R$--valued r.v.'s $U^\epsilon_0$ is tight. We only need check tightness in $C([0,\infty))$
of the process $W^\epsilon_t:=\int_0^te^{2s}dM^\epsilon_s$. Since $W^\epsilon_0=0$, we need only verify condition (ii) from Theorem 7.3 in Billingsley \cite{CPM_2}, which follows from the condition of the Corollary of Theorem 7.4 again in \cite{CPM_2}. In other words it suffices to check that for any $T$, $\eta$ and $\eta'>0$, there exists $\delta\in(0,1)$
such that for all $\epsilon>0$, $0\le t\le T-\delta$,
\begin{equation}\label{condBIL}
\frac{1}{\delta}\P\left(\sup_{t\le s\le t+\delta}\left|W^\epsilon_s-W^\epsilon_t\right|\ge\eta\right)\le\eta'.
\end{equation}
Combining Chebycheff and Burkholder--Davis--Gundy inequalities, we deduce that
(we use below the result from Lemma \ref{moments})
\begin{align*}
\P\left(\sup_{t\le s\le t+\delta}\left|W^\epsilon_s-W^\epsilon_t\right|\ge\eta\right)&\le
\eta^{-4}\E\left(\left|\langle W^\epsilon\rangle_{t+\delta}-\langle W^\epsilon\rangle_{t}\right|^2\right)\\
&\le \eta^{-4}e^{8T} C^2\delta\int_t^{t+\delta}\E[(U^\epsilon_s+V^\epsilon_s)^2]ds\\
&\le\eta^{-4}\bar C e^{8T} \delta^2,
\end{align*}
from which \eqref{condBIL} follows if we choose $\delta=e^{-4T}\eta^2\sqrt{\eta'/\bar C}$.
\end{proof}

\subsection{Tighness of $\lambda^\epsilon$}
\label{sec:tighness-lambda}
Since $(U_t^\epsilon,V_t^\epsilon)$ is not a Markov process it does
not have an invariant probability measure. However the projection
$\lambda^\epsilon=\mu^\epsilon \Phi^{-1}$ of $\mu^\epsilon$, the unique invariant probability measure of
the Markov process $\xi_t^\epsilon$, is well defined. We now establish the following
tightness result:
\begin{lemma}\label{tightnessLambda} The sequence of measure $\{ \lambda^\epsilon: \epsilon
  >0\}$ is tight on the space $(0,\infty)\times(0,\infty)$.
\end{lemma}

\begin{remark} We emphasis that Lemma~\ref{tightnessLambda} is
  tightness in the open set $(0,\infty)\times(0,\infty)$ which implies
  the measure does not accumulate neither at the boundary at
  ``infinity'' nor at the boundary at zero.
  In other words, for any $\delta>0$ there exists a $r>0$ so that
  \begin{align*}
     \inf_{\epsilon>0} \, \lambda^\epsilon ( [\tfrac1r,r] \times
     [\tfrac1r,r]) > 1-\delta.
  \end{align*}
\end{remark}

The following result which implies the tightness at infinity follows
immediately from the definition of $\lambda^\epsilon$, the definition
of $\Phi$ and Corollary~\ref{cor:xiInvMeasure}.
\begin{lemma}\label{LambdaBoundAbove}
  For any $p\geq 1$, there exists a $C(p)>0$ so that
  \begin{align*}
    \sup_{\epsilon >0}    \int (u^{p} + v^p) \lambda^\epsilon(du, dv) < C(p)
  \end{align*}
\end{lemma}

We now handle the boundary at zero.

\begin{lemma}\label{l:tightNearZero}Let $\zeta_t$ be a Markov process and $f$ and $g$ two real-valued functions on the state space of $\zeta_t$ satisfying $0\leq g(\zeta_t) \leq f(\zeta_t)$ for all $t \geq0$ almost surely and such that $f(\zeta_t)$ is a continuous semimartingale satisfying \begin{equation*}
  df(\zeta_t) = (a-f(\zeta_t))dt + c \sqrt{g(\zeta_t)} dW_t
\end{equation*}
where $a$ and $c$ are positive constants and $W_t$ a standard Wiener process.
If $\mu$ is any invariant probability measure of $\zeta_t$ with $\mu[\, f^2\,]=\int f^2(\zeta) d\zeta < \infty$, then for any $\delta \in (0,1)$.
\begin{align}
    \mu\{ f  \leq \delta\} & \leq  \frac{ \mu \big[\, f^2\, \big]+b}{ a\,|\log \delta|} \,,
\end{align}
with $b=1+c^2/2$.
\end{lemma}
\begin{proof}[Proof of Lemma~\ref{l:tightNearZero}]
Defining
\begin{align*}
  \phi(x) &=
  \begin{cases}
    -\frac1x ,& x \leq 1\\
    1 ,& x\geq 1
  \end{cases}, \quad     I(x) =
  \begin{cases}
1-\log x  ,& x \leq 1\\
x ,& x \geq 1
  \end{cases},\\ H(x)&=
  \begin{cases}
    2x - x \log x ,& x \leq 1\\
    \frac12 x^2+\frac32 ,& x\geq 1
  \end{cases}.
\end{align*}
Observe that $ H'(x) = I(x) $ and $I'(x)=\phi(x)$ and that $\phi$, $I$ and $H$ are well defined on the
intervals $(0,\infty)$, $(0,\infty)$ and $[0,\infty)$ respectively. $H$ and $I$ are everywhere positive, while $\phi$ is positive on $[1,+\infty)$ and negative on $(0,1)$. It is plain that the discontinuity of $H''$ at $x=1$ will not prevent us from using It\^o's formula.
Taking $\zeta_0$ distributed according to $\mu$, noticing that since
$H(x) < 2+ x^2$ for $x\geq 0$, and setting $X_t=f(\zeta_t)$ for
notational convenience, we have that
\begin{align}\label{eq:HFinite}
  \mu \big[\, \E_{\zeta_0} H(X_t) \,\big]= \mu \big[\, \E_{\zeta_0}
  (H\circ f)(\zeta_t) \,\big] =\mu \big[\, H\circ f\,\big] < \infty.
\end{align}
Now from It\^o's formula
\begin{align*}
  d H(X_t)= (a-X_t)I(X_t) dt + \frac12 c^2 g(\zeta_t)\phi(X_t)dt + dM_t
\end{align*}
where $M_t$ is the Martingale defined by $dM_t = c \sqrt{g(\zeta_t)} I(X_t) d W_t$.
We conclude that
\begin{align*}
  a\int_0^t \E_{\zeta_0}[I(X_s)]\,ds \leq \E_{\zeta_0} &H(X_t) -H(X_0)
  + \int_0^t \E_{\zeta_0}\, X_s I(X_s) \, ds \\
  &+ \frac{c^2}{2}\int_0^t \E_{\zeta_0}\, g(\zeta_s)\phi^-(X_s)  \, ds\,.
\end{align*}
Now integrating over the initial conditions $\zeta_0$ (which were
distributed according to $\mu$), we see that $H$ terms are equal by
the stationarity embodied in \eqref{eq:HFinite} (and hence they
cancel) and that
\begin{align*}
a\, \mu[I\circ f] \leq \mu [ f (I\circ f)] + \frac{c^2}{2}\mu [g(\phi^-\circ f)]
\end{align*}
and since $g(\phi^-\circ f)\le1$,
\begin{align*}
  a\, \mu\big[ \,|\log f | \,\mathbf{1}\{f \leq 1\}\,\big] \leq a\,
  \mu[\, I\circ f\, ] \leq \,\mu [ \, f (I\circ f)\,]+ \frac{c^2}{2}
\end{align*}
Finally, for any $\delta \in (0,1)$
\begin{align*}
  a \,\mu\{ f \leq \delta\} &=a\, \mu \big[\, \{|\log f| \geq |\log
  \delta|\} \cap \{f \leq 1\}\,\big] \\&\leq \frac{a\,\mu\big[\,|\log
    f | \,\mathbf{1}\{f \leq 1\}\,\big]}{|\log \delta|} \leq \frac{
    \mu \big[\, f (I\circ f)\, \big]+ \frac{c^2}{2}}{|\log \delta|}
\end{align*}
The result follows, since $xI(x)\le1+x^2$.
\end{proof}

The following Corollary is a direct consequence of the two last Lemmata
\begin{corollary}\label{cor:zero}
There exists a constant $C>0$ so that for any
  $\delta\in(0,1)$
  \begin{align*}
    \sup_{\epsilon>0} \lambda^\epsilon \{ (u,v) : u + v < \delta\}
    \leq \frac{C}{|\log \delta|}
  \end{align*}
\end{corollary}

\begin{proof}[Proof of Lemma~\ref{tightnessLambda}]
  The result follows immediately by combining  \\
  Lemma~\ref{LambdaBoundAbove}
  and Corollary \ref{cor:zero}.
\end{proof}

\subsection{Convergence of Quadratic variation}
\label{sec:conv-quadr-vari}
Now that we know that the collection $\{(U^\epsilon_t,V^\epsilon_t),\,
t\ge0\}_{\epsilon>0}$ is tight,
in view of Theorem \ref{th:uniq}, the weak uniqueness result for \eqref{UV},
and comparing \eqref{rho-chi} and \eqref{UV}, the weak convergence
$(U^\epsilon,V^\epsilon)\Rightarrow(U,V)$ will follow from the convergence of
the quadratic variations of $U^\epsilon$ and $V^\epsilon$ to those of $U$ and
$V$, which will be proved in the next Lemma.

For each $M>0$, let
$$\stop:=\inf\{t>0,\ U^\epsilon_t\vee V^\epsilon_t>M\}.$$

Considering the three different cases of the behavior of $(U,V)$, it
is not hard to see that in all cases $\stopL$, defined exactly as
$\stop$, but with $(U^\epsilon,V^\epsilon)$ replaced by $(U,V)$, is
a.s. a continuous function of the $(U,V)$ trajectory, hence
$$\stop\Longrightarrow \stopL\quad\text{as }\epsilon\to0$$
will follow from $(U^\epsilon,V^\epsilon)\Rightarrow(U,V)$.

In particular
$$\liminf_{\epsilon\to0}\P(\stop>t)\ge\P(\stopL>t).$$
Clearly for all $t>0$,
$$\P(\stopL>t)\to1,\quad\text{as } M\to\infty.$$
It will then follow that for any $t>0$, the lim inf as $\epsilon\to0$ of $\P(\stop>t)$ can be made arbitrarily close to 1, by choosing $M$ large enough.

\begin{lemma}\label{conv2var}
Let $\nu_\varepsilon$ be any sequence of tight probability
  measures on $\R^3$ and let $(\widetilde{X}_t^\varepsilon,\widetilde{Y}_t^\varepsilon,\widetilde{Z}_t^\varepsilon)$
  be the solution to \eqref{fast} with
  $(\widetilde{X}_0^\varepsilon,\widetilde{Y}_0^\varepsilon,\widetilde{Z}_0^\varepsilon)$ distributed as
  $\nu_\varepsilon$. Then for any $t>0$,  as $\epsilon \rightarrow 0$,
  \begin{align*}
    \int_0^t (\widetilde{X}_s^\varepsilon)^2ds  \Rightarrow\!\! \int_0^t \avg{ x^2}({U_s,V_s})ds\ \text{and}\
 \int_0^t (\widetilde{Y}_s^\varepsilon)^2ds \Rightarrow\!\!   \int_0^t \avg{ y^2}({U_s,V_s})ds.
  \end{align*}
\end{lemma}
  \begin{proof}
   Since
$2(\widetilde{X}_s^\varepsilon)^2=U^\varepsilon_s-(\widetilde{Z}_s^\varepsilon)^2$
and
$2(\widetilde{Y}_s^\varepsilon)^2=V^\varepsilon_s-(\widetilde{Z}_s^\varepsilon)^2$,
we only need to show that $\int_0^t (\widetilde{Z}_s^\varepsilon)^2ds
\Rightarrow \int_0^t \avg{ z^2}({U_s,V_s})ds$.  It suffices in fact to
show that
\begin{align*}
  \int_0^{t\wedge\stop} (\widetilde{Z}_s^\varepsilon)^2ds  \Longrightarrow
\int_0^{t\wedge\stopL} \avg{ z^2}({U_s,V_s})ds\ ,
\end{align*}
for all $M>0$.

The vlaues of $t>0$ and $M$ will be fixed throughout this proof.
For any $\delta>0$, we define $N_\delta=\lceil t/\delta\rceil$, $t_n=n\delta\wedge\stop$ for $0\le n<N_\delta$
and $t_{N_\delta}=t\wedge\stop$.
Let now $Z_s^{(n)}$ be the $z$ component of the solution
to the deterministic dynamics \eqref{epsZero} at time $s$ which started at time $t_{n}$ from the
point
$(X_{t_{n}}^\epsilon,Y_{t_{n}}^\epsilon,Z_{t_{n}}^\epsilon)$. Then clearly
\begin{equation}\label{split}
\begin{split}
  \int_0^{t\wedge\stop} (\widetilde{Z}_s^\varepsilon)^2ds  = \epsilon\sum_{n=0}^{N_\delta-1}
  \int_{t_{n}/\epsilon}^{t_{n+1}/\epsilon} (Z_s^\epsilon)^2 ds= \Phi_{\epsilon,\delta}+
  \Xi_{\epsilon,\delta}
\end{split}
\end{equation}
where
\begin{align*}
\Phi_{\epsilon,\delta}&= \delta\sum_{n=0}^{N_\delta-1}\frac{\epsilon}{\delta}
  \int_{t_{n}/\epsilon}^{t_{n+1}/\epsilon} (Z_s^{(n)})^2 ds \\
  \Xi_{\epsilon,\delta}&=\epsilon\sum_{n=0}^{N_\delta-1}
  \int_{t_{n}/\epsilon}^{t_{n+1}/\epsilon} \left[(Z_s^\epsilon)^2-(Z_s^{(n)})^2\right] ds.
\end{align*}

 To control the error term observe that
 \begin{align*}
   |\Xi_{\epsilon,\delta}|\le\sqrt{\epsilon\sum_{n=0}^{N_\delta-1}
  \int_{t_{n}/\epsilon}^{t_{n+1}/\epsilon} \left[Z_s^\epsilon+Z_s^{(n)}\right]^2 ds}
  \sqrt{\epsilon\sum_{n=0}^{N_\delta-1}
  \int_{t_{n}/\epsilon}^{t_{n+1}/\epsilon} \left[Z_s^\epsilon-Z_s^{(n)}\right]^2 ds}\,.
 \end{align*}
The first term in the product  on the righthand side is
bounded due to the stopping time $\stop$. Using Lemma~\ref{rhoBound},
we see that
$\E|\Xi_{\eps,\delta}|$ is bounded by a constant times the square root of
\begin{align*}
  \epsilon\sum_{n=0}^{N_\delta-1}
  \int_{t_{n}/\epsilon}^{t_{n+1}/\epsilon} \E\left[Z_s^\epsilon-Z_s^{(n)}\right]^2 ds
  \le C_M \epsilon^2\lceil\frac{t}{\delta}\rceil\exp\left[C_M\frac{\delta}{\epsilon}\right]\,.
\end{align*}
Hence if we choose
  \begin{equation}\label{delta}
  \delta=C_M^{-1}\epsilon\log(1/\epsilon),
  \end{equation}
   then $\Xi_{\epsilon,\delta}\to0$ in $L^1(\Omega)$ as
   $\epsilon\to0$. Having made this choice of $\delta$, we now
   suppress it from notation designating dependence on parameters.

   We now further divide $\Phi_{\epsilon}$ ($\delta$ having been
   suppressed) depending on whether in phase space the starting point
   $(X_{t_n/\epsilon}^{(n)}
   ,Y_{t_n/\epsilon}^{(n)},Z_{t_n/\epsilon}^{(n)})$ lies the region
   where $|U_{t_n}^\epsilon - V_{t_n}^\epsilon|$ is small or not.  To accomplish this,
   for any $\rho>0$, let $\chi_\rho\in C(\R;[0,1])$ be such that
\begin{align*}
  \chi_\rho(x)=\begin{cases}
    0&,\text{if $|x|\ge\rho$},\\
    1&,\text{if $|x|\le\rho/2$}
  \end{cases}
\end{align*}
and define $\bar\chi_\rho=1- \chi_\rho$. Consider the decomposition
$\Phi_{\epsilon}= A_{\epsilon,\rho}+B_{\epsilon,\rho}$ where
\begin{align*}
  A_{\epsilon,\rho}=& \delta\sum_{n=0}^{N_\delta-1}\frac{\epsilon}{\delta}
\chi_\rho(U_{t_n}^\epsilon -V_{t_n}^\epsilon)   \int_{t_{n}/\epsilon}^{t_{n+1}/\epsilon} (Z_s^{(n)})^2 ds,\\
B_{\epsilon,\rho}=& \delta\sum_{n=0}^{N_\delta-1}\frac{\epsilon}{\delta}
  \bar\chi_\rho(U_{t_n}^\epsilon -V_{t_n}^\epsilon) \int_{t_{n}/\epsilon}^{t_{n+1}/\epsilon} (Z_s^{(n)})^2 ds\,.
\end{align*}
The reason for this decomposition is that why the time average
of $(Z^{n})^2$ over the time interval $[t_{n+1},t_n]$  is close to the
function $\mathcal{A}(z^2)(U_{t_n},V_{t_n})$ is different in
the two regions. The terms which have $|u-v|>\rho$ have periods
uniformly bounded from above and hence as $(t_{n+1}-t_n)/\epsilon=\delta/\epsilon \rightarrow
\infty$ the number of periods contained in the interval over which we
are averaging also goes to infinity. On the other hand, as the points
approach the diagonal $u=v$ the period grows to infinity. So for
$|u-v|$ small enough the period might be  much greater  than
the length of the time interval $(t_{n+1}-t_n)/\epsilon=\delta/\epsilon$ over which we are
averaging. Hence the reason for convergence for the
$A_{\epsilon,\rho}$ to the appropriate average values occurs by a
different mechanism. Proposition~\ref{avgNearDiag} shows that
$\mathcal{A}(z^2)(u,v) \rightarrow u=v$ as $|u-v| \rightarrow
0$. To understand why
$\frac{\epsilon}{\delta}\int_{t_{n}/\epsilon}^{t_{n+1}/\epsilon}
(Z_s^{(n)})^2ds \rightarrow U_{t_n} \wedge V_{t_n}$ one needs to
recall the discussion from Section~\ref{sec:detDynam}. The
deterministic orbits when $u=v$ consist of heteroclinic orbits
connecting the fixed points at $(0,0,\sqrt{u})$ and $(0,0,-\sqrt{u})$.
Since the time to reach the fixed points on these orbits in infinite, it
is not surprising that for $|u-v|$ small the periodic orbit spends
most of its time near $(0,0,\pm\sqrt{u})\sim
(0,0,\pm\sqrt{v})$. This can also be seen in the fact that the
occupation measures given in \eqref{ocupation} concentrates around
$\theta\sim \pi/2, 3\pi/2$, which corresponds to the fixed points, if $|u-v|\sim 0$.
Importantly, even when the time is not long enough
to traverse the orbit completely, any average will be concentrated
near the fixed points since the time to reach the neighborhood of the fixed
point is small relative to the time it will take to leave that
neighborhood once it has arrived there. This idea will be made
quantitative below.

Hence we define
\begin{align*}
   \hat A_{\epsilon,\rho}&= \delta\sum_{n=0}^{N_\delta-1}
\chi_\rho(U_{t_n}^\epsilon -V_{t_n}^\epsilon) \,  (U_{t_n}^\epsilon \wedge
V_{t_n}^\epsilon ) \quad\text{and}\quad \Theta_{\epsilon,\rho}=
A_{\epsilon,\rho}-\hat  A_{\epsilon,\rho}\\
\hat B_{\epsilon,\rho}&=\delta \sum_{n=0}^{N_\delta-1}\bar
\chi_\rho(U_{t_n}^\epsilon -V_{t_n}^\epsilon) \gamma_\epsilon(U^\epsilon_{t_n},V^\epsilon_{t_n})  \mathcal{A}(z^2)(U^\epsilon_{t_{n}},
V^\epsilon_{t_{n}})  \quad\text{and}\quad \Upsilon_{\epsilon,\rho}=B_{\epsilon,\rho} -\hat B_{\epsilon,\rho}
\end{align*}
where
\begin{align*}
   \gamma_\epsilon(u,v) &= \bigg\lfloor \frac{\delta}{\epsilon
     \tau(u,v)}\bigg\rfloor  \frac{\epsilon \tau(u,v)}{\delta}\,,
   \end{align*}
and  $\tau(u,v)$ was the period of the deterministic orbit.

  For all $\beta >0$ and $\alpha >\rho$, we define
\begin{align*}
  \hat \tau_{\beta,\rho} &= \min\Big\{ \tau(u,v)  :  |u-v| \leq \rho ,
  \beta \leq u\wedge v \leq u \vee v \leq M\Big\}, \\
  \Psi_{\alpha,\rho}&= 4 \sqrt{M}\,\int_0^{1-\alpha}
  \frac{1}{|1-a^2|} da \, .
\end{align*}
The utility of $\Psi$ is the following which can be deduced from Section~\ref{avgs}
\begin{align*}
  \sup_{u,v \leq M} \text{Leb}\left(\{s>0,\, |z_s| \not \in [(1-\alpha)
    \sqrt{u\wedge v},\sqrt{u\wedge v}] \}\right) \leq  \Psi_{\alpha,\rho}\,
  .
\end{align*}
Now
\begin{align*}
  | \Theta_{\epsilon,\rho}| \leq t \Big[ M \big(
  \frac{\Psi_{\alpha,\rho}}{\delta/\epsilon} +
    \frac{\Psi_{\alpha,\rho}}{ \hat \tau_{\beta,\rho}} \big) + 2 \beta
    \Big]\eqdef K_{\epsilon,\rho,\alpha,\beta}\, .
\end{align*}
On the other hand, since
 \begin{align*}
\Upsilon_{\eps,\rho}&=\epsilon\sum_{n=0}^{N_\delta-1}
\bar\chi_\rho(U_{t_n}^\epsilon -V_{t_n}^\epsilon)
\int_{\frac{t_n}{\epsilon}+\left[\frac{\delta}{\epsilon\tau_n^\epsilon}\right]\tau_n^\epsilon}^{\frac{t_{n+1}}{\epsilon}}(Z^{(n)}_r)^2 dr\, ,
\end{align*}
we have the inequality
\begin{align*}
 |\Upsilon_{\eps,\rho}|&\le M\frac{\epsilon}{\delta}\ \delta\sum_{n=0}^{N_\delta-1}
 \bar\chi_\rho(U_{t_n}^\epsilon -V_{t_n}^\epsilon)
\tau(U^\epsilon_{t_n},V^\epsilon_{t_n})\\
&\le \frac{\epsilon}{\delta}\bar M(\rho)\eqdef L_{\epsilon,\rho}\,,
\end{align*}
where $\bar M(\rho):=\sup_{u,v\le M}\bar\chi_\rho(u,v)\tau(u,v)<\infty$ if $\rho>0$. Hence $L_{\epsilon,\rho}\to0$ as $\epsilon\to0$, for any $\rho>0$.

Now  applying  Lemma~\ref{convunif} below, we see that for all $\rho>0$, as $\epsilon \rightarrow 0$
one has
\begin{multline*}
  \hat A_{\epsilon,\rho} + \hat B_{\epsilon,\rho} \Longrightarrow
  \hat A_{\rho} + \hat B_{\rho}\eqdef\int_0^{t
    \wedge \kappa_M} \chi_\rho(U_s-V_s)\, ( U_s\wedge V_s)
  ds \\+ \int_0^{t \wedge \kappa_M}\bar \chi_\rho(U_s -V_s)
 \, \mathcal{A}(z^2)(U_s,V_s) ds
\end{multline*}
Notice that as $\rho \rightarrow 0$, $\hat A_{\rho} +\hat  B_{\rho}$
converges to
\begin{align*}
\int_0^{t    \wedge \kappa_M} \one_{U_s=V_s}\, ( U_s\wedge V_s)
  ds + \int_0^{t \wedge \kappa_M} \one_{U_s\neq V_s}\, \mathcal{A}(z^2)(U_s,V_s) ds\,.
\end{align*}
By Proposition~\ref{avgNearDiag}, we see that $\one_{U_s=V_s}\,
( U_s\wedge V_s)= \one_{U_s=V_s}\,   \mathcal{A}(z^2)(U_s,V_s)$ since
$(x,y,z) \mapsto z^2$ evaluated at $(0,0,\sqrt{u\wedge v})$ is $u
\wedge v$. (Of course $u
\wedge v=u=v$ since we are considering the case $u=v$.)
In light of this, we
conclude that $\hat A_{\rho} +\hat B_{\rho}$ converges to
\begin{align*}
   \int_0^{t \wedge \kappa_M} \mathcal{A}(z^2)(U_s,V_s) ds
\end{align*}
as $\rho \rightarrow 0$.

Now let $F \in C(\R^+,\R^+)$ be any increasing, bounded
function. Then
\begin{multline*}
  F(\hat A_{\epsilon,\rho} + \hat B_{\epsilon,\rho} -
  K_{\epsilon,\rho,\alpha,\beta}-L_{\epsilon,\delta}) \leq F(A_{\epsilon,\rho} +
  B_{\epsilon,\rho}) \leq   \\F(\hat A_{\epsilon,\rho} + \hat B_{\epsilon,\rho} +
  K_{\epsilon,\rho,\alpha,\beta}+L_{\epsilon,\delta}).
\end{multline*}
Observe that $A_{\epsilon,\rho} +
  B_{\epsilon,\rho}=\Phi_\epsilon$ and hence is independent of the
  choice of $\rho$. Since $F$ is bounded and as already noted $L_{\epsilon,\rho}
  \rightarrow 0$ as $\epsilon \rightarrow0$ for any $\rho>0$, we have that
  \begin{align*}
\E F(    \hat A_{\rho} +\hat  B_{\rho} -
  K_{\rho,\alpha,\beta}) &\leq \varliminf_{\epsilon \rightarrow 0} \E
  F(\Phi_\epsilon) \\&\leq \varlimsup_{\epsilon \rightarrow 0} \E
  F(\Phi_\epsilon) \leq  \E F(    \hat A_{\rho} +\hat B_{\rho} +
  K_{\rho,\alpha,\beta}),
 \end{align*}
where
\begin{align*}
   K_{\rho,\alpha,\beta}=\lim_{\epsilon \rightarrow0}
   K_{\epsilon,\rho,\alpha,\beta} =  t \Big[ M \big(
    \frac{\Psi_{\alpha,\rho}}{ \hat \tau_{\beta,\rho}} \big) + 2 \beta
    \Big]\,.
\end{align*}
Now since $K_{\rho,\alpha,\beta} \rightarrow 0$ if $\rho \rightarrow
0$, followed by $\alpha \rightarrow 0$, followed by $\beta
\rightarrow0$ we obtain that  $\lim_{\eps\to0}\E
  F(\Phi_\epsilon)$ exists and equals
  $\lim_{\rho\to0}\E F(    \hat A_{\rho} +\hat B_{\rho})$.

It remains to exploit Lemma \ref{BV} below to deduce that
$$\Phi_\epsilon\Rightarrow \int_0^{t \wedge \kappa_M} \mathcal{A}(z^2)(U_s,V_s) ds$$
as $\epsilon\to0$.
  \end{proof}

\begin{lemma}\label{convunif}
  Let $X_n$ be a sequence of $\mathcal{X}$--valued r. v.'s, and $X$ be
  such that $X_n\Rightarrow X$, where $\mathcal{X}$ is a separable
  Banach space. Let $\{F_n,\ n\ge1\}$ be a sequence in
  $C(\mathcal{X})$, which is such that as $n\to\infty$, $F_n\to F$
  uniformly on each compact subset of $\mathcal{X}$. Then
  $F_n(X_n)\Rightarrow F(X)$, as $n\to\infty$.
\end{lemma}
\begin{proof}[Proof of Lemma~\ref{convunif}] Choose $\epsilon>0$
  arbitrary, and let $K$ be a compact subset of $\mathcal{X}$ such
  that $\P(X_n\not\in K)\le\epsilon$, for all $n\ge1$. Now choose $n$
  large enough such that $|F_n(x)-F(x)|\le\epsilon$, for all $x\in
  K$. Choose an arbitrary $G\in C_b(\R)$, such that
  $\sup_x|G(x)|\le1$.  We have
\begin{align*}
|\E[G\circ F_n(X_n)]-\E[G\circ F(X)|&\le|\E[G\circ F_n(X_n)-G\circ F(X_n);X_n\in K]|
\\&\quad+2\epsilon+|\E[G\circ F(X_n)-G\circ F(X)]|
\end{align*}
 The first term  of the righthand side can be made arbitrarily small by choosing $\epsilon$ small,
 uniformly in $n$, since $G$ is uniformly continuous on the union of the images of $K$ by the $F_n$'s.
The last term clearly goes to zero as $n\to\infty$.
\end{proof}

\begin{lemma}\label{BV}
  Let $\{X_n,\, n\ge1\}$ and $X$ denote real--valued random variables,
  defined on a given probability space $(\Omega,\F,\P)$. A sufficient
  condition for $X_n\Rightarrow X$ is that
$$\E[F(X_n)]\to\E[F(X)],$$ for any continuous, bounded and increasing function $F$.
\end{lemma}
\begin{proof}[Proof of Lemma~\ref{BV}]
  It is plain that the condition of the Lemma implies that
  $\E[F(X_n)]\to\E[F(X)]$ for any $F$ continuous, bounded with bounded
  variations. Associating to each $M>0$ a continuous function $F_M$
  from $\R$ into $[0,1]$, which is decreasing on $\R_-$ and increasing
  on $\R_+$, equal to zero on the interval $[-M+1,M-1]$, and to one
  outside the interval $[-M,M]$, we note that the condition of the
  Lemma implies that
\begin{align*}
\limsup_{n\to\infty}\P(|X_n|>M)&\le
\lim_{n\to\infty}\E[F_M(X_n)]\\ &=\E[F_M(X)].
\end{align*}
Since the last right--hand side can be made arbitrarily small by
choosing $M$ large enough, the last statement implies tightness of the
sequence $\{X_n,\, n\ge1\}$. Consequently $X_n\Rightarrow X$ will
follow if $\E[F(X_n)]\to\E[F(X)]$ for any $F$ in a class of continuous
and bounded functions which separates probability measures, which
clearly is the case under the condition of the theorem.
\end{proof}

\section{Proof of Theorem \ref{limitinvmeas}}\label{sec9}
Recall that for each $\epsilon>0$ $\mu^\epsilon$ denotes the unique
  invariant probability measure of $\xi^\epsilon$, and that $\lambda$ denotes the unique invariant probability measure of the
 diffusion process $(U_t,V_t)$ or equivalently of its semigroup $Q_t$.

The fact that any accumulation point of the collection $\{\mu^\epsilon,\, \epsilon>0\}$
satisfies $\mu=\mu\sym_\pm^{-1}$ follows from Proposition
\ref{symInvMeasure}. Corollary~\ref{cor:uniqRepresentationThm} states
that at most one invariant probability measure of \eqref{epsZero} satisfies both
$\mu=\mu\sym_\pm^{-1}$ and $\mu\Phi^{-1}=\lambda$.

   From Corollary~\ref{cor:xiInvMeasure} we know the collection
$\{\mu^\epsilon,\, \epsilon>0\}$ is tight. Consequently, there exists a
sequence $\epsilon_n\to0$ and a measure $\tilde \mu$, such that $\mu^{\en}\Rightarrow\tilde\mu$.

We now show that this $\tilde \mu$ is invariant for the $\xi_r$ dynamics. Fix an arbitrary $t>0$. If we initialize $\xi^\en$ with its invariant probability measure $\mu^\en$, then
both marginal laws of the pair $(\xi^\en_0,\xi^\en_t)$ equal $\mu^\en$. Since
$(\xi^\en_0,\xi^\en_t)\Rightarrow(\xi_0,\xi_t)$,
 we deduce that if
$\xi_0\simeq \tilde \mu$, then
$\xi_t\simeq\tilde\mu$, and this is true for all $t>0$, hence $\tilde\mu$ is invariant for $\xi$.

Next recall that for each $\epsilon>0$, we defined
$\lambda^\epsilon:=\mu^\epsilon \Phi^{-1}$. Since both marginal laws of the pair
$(\xi^\en_0,\xi^\en_{t/\epsilon})$ equal $\mu^\en$, we conclude that both marginal laws
of the pair $((U^\en_0,V^\en_0),(U^\en_t,V^\en_t))$ equal
$\lambda^\en$.
We  now show that  $\tilde \lambda=\tilde\mu\Phi^{-1}$ is an
invariant measure for $Q_t$. Combining the fact that
\begin{align*}
  ((U^\epsilon_0,V^\epsilon_0),(U^\epsilon_t,V^\epsilon_t))\Rightarrow((U_0,V_0),(U_t,V_t))\,,
\end{align*}
$\Phi$ is
continuous and that from Lemma~\ref{tightnessLambda} we know that
$\tilde\lambda$ is supported on
$(0,+\infty)\times(0,+\infty)$,
we conclude that $\lambda^\en\Rightarrow\tilde\lambda$. Since the marginals are equal for all
$t>0$, we conclude that $\tilde \lambda$ is an invariant probability measure for
$Q_t$. Since from  Theorem~\ref{thm:invMeasureUV}, $Q_t$ has the unique
invariant probability measure $\lambda$, we conclude that
$\tilde \lambda = \lambda$.

Hence $\tilde\mu=\mu$, and $\mu^\epsilon\Rightarrow\mu$, as $\epsilon\to0$.

Because $\lambda=\mu\Phi^{-1}$ does no charge the diagonal, $\mu$ does not charge the set
$\{x=y\}\cup\{x=-y\}$.
Pick a point $(x,y,z)\in\R^3$, with $|x|\not=|y|$ and let $(u,v)=\Phi(x,y,z)$. Assume that
$(x,y,z)\in\Gamma^+_{u,v}$ (the other case is treated exactly in the same way). To $(x,y,z)$
corresponds a value $\theta$ of the parameter on  $\Gamma^+_{u,v}$ defined in Section~\ref{sec:avg}.
There exists a smooth bijection $\Psi$ with a smooth inverse from an open neighborhood $\mathcal{O}$ of
$(x,y,z)$ onto an open neighborhood $\mathcal{U}$ of $(u,v,\theta)$. The restriction of $\mu$ to $\mathcal{O}$
is the image by $\Psi^{-1}$ of the restriction to $\mathcal{U}$ of the measure
$$\frac{K_{u,v}}{|\cos(\phi_{u,v}(\theta)|}\rho(u,v)d\theta dudv,$$
where $\rho$ denotes the density of $\lambda$ with respect to Lebesgue measure and we have used
formula \eqref{ocupation}. Hence the restriction of $\mu$ to $\mathcal{O}$  is absolutely continuous with respect to Lebesgue
measure on $\R^3$, with a density which is positive at $(x,y,z)$ since $\rho(u,v)>0$.

\vspace{1em}
\section*{Acknowledgments}
The authors are  indebted to two anonymous referees who pointed some errors in the original version of this paper, to Alexander Veretennikov for
bringing \cite{Por} to our attention, and to Hans--Juergen Engelbert, for discussions on uniqueness/non--uniqueness for
one--dimensional SDEs.   JCM thanks the National Science
Foundation for its support through the grant NSF-DMS-08-54879 (FRG).

\end{document}